\documentclass{article}

\usepackage{mathrsfs}
\usepackage[sans]{dsfont}
\usepackage[mathscr]{euscript}
\usepackage{amsmath,amssymb,amsthm}
\usepackage[british]{babel}
\usepackage[nospace]{cite}

\newcommand{\Gcal}{\mathcal{G}}

\newcommand{\Kcal}{\mathcal{K}}
\newcommand{\Lcal}{\mathcal{L}}

\newcommand{\Tcal}{\mathcal{T}}

 \newcommand{\Dscr}{\mathscr{D}}
 \newcommand{\Escr}{\mathscr{E}}
\newcommand{\Fscr}{\mathscr{F}}

\newcommand{\Hscr}{\mathscr{H}}

\newcommand{\Lscr}{\mathscr{L}}
 \newcommand{\Mscr}{\mathscr{M}}

\newcommand{\Sscr}{\mathscr{S}}
\newcommand{\Tscr}{\mathscr{T}}
\newcommand{\Uscr}{\mathscr{U}}
\newcommand{\Zscr}{\mathscr{Z}}

\newcommand{\rmd}{\mathrm{d}}
\newcommand{\rme}{\mathrm{e}}
\newcommand{\rmi}{\mathrm{i}}
\newcommand{\Real}{\mathbb{R}}

\newcommand{\mr}{\mathrm}

\newcommand{\1}{\mathds{1}}

\newcommand{\ra}{\rangle}
\newcommand{\la}{\langle}

\newcommand{\LII}{L^2(\mathbb{R}_+;\mathscr{Z})}
\newcommand{\til}{\widetilde}

\newcommand{\RE}{\operatorname{Re}}
\newcommand{\IM}{\operatorname{Im}}
\newcommand{\Tr}{\operatorname{Tr}}

\newcommand{\abs}[1]{\left\vert#1\right\vert}
\newcommand{\norm}[1]{\left\Vert#1\right\Vert}

\newcommand{\Dom}{\mathrm{Dom}}
\theoremstyle{plain}
\newtheorem{theorem}{Theorem}
\newtheorem{corollary}[theorem]{Corollary}
\newtheorem{lemma}[theorem]{Lemma}
\newtheorem{proposition}[theorem]{Proposition}
\theoremstyle{definition}
\newtheorem{definition}{Definition}
\newtheorem{remark}{Remark}
\newtheorem{hypo}{Hypothesis}%[section]

\begin{document}

%\markboth{R. Castro Santis, A. Barchielli}{Continuous measurements: unbounded coefficients}

\title{Quantum stochastic differential equations and continuous measurements: unbounded coefficients}

\author{Ricardo Castro Santis\footnote{RCS is supported in part by FONDECYT Grants 3090055 and Bicentennial
Foundation Grants PBCT-ACT-13 ``Laboratorio de An\'alisis Estoc\'astico''}
\\
Centro de An\'alisis Estoc\'astico y Aplicaciones,\\ Pontificia Universidad Cat\'olica de Chile\\
Avda. Vicu\~na Mackenna 4860, Macul, Santiago Chile; \\ ricardo.castro@anestoc.cl
\\ \\
Alberto Barchielli\footnote{Also: Istituto Nazionale di Fisica Nucleare, Sezione di Milano}
\\
Politecnico di Milano, Dipartimento di Matematica\\
Piazza Leonardo da Vinci 32, I-20133 Milano, Italy\\
Alberto.Barchielli@polimi.it}

\maketitle

\begin{abstract}
A natural formulation of the theory of quantum measurements in continuous time is based on
quantum stochastic differential equations (Hudson-Parthasarathy equations). However, such a
theory was developed only in the case of Hudson-Parthasarathy equations with bounded
coefficients. By using some results on Hudson-Parthasarathy equations with unbounded
coefficients, we are able to extend the theory of quantum continuous measurements to cases in
which unbounded operators on the system space are involved. A significant example of a quantum
optical system (the degenerate parametric oscillator) is shown to fulfill the hypotheses
introduced in the general theory.

\end{abstract}

\section{Introduction}

A powerful formulation of the quantum theory of measurements in continuous time is based on
quantum stochastic calculus \cite{Hud-Partha,Partha}. In such an approach, the quantum
stochastic Schr\"odinger equation, or Hudson-Parthasarathy equation (HP-equation), is combined
with suitable field observables \cite{BarL85JMP,Bar86PR,BarSpringer,Bel88,Belavkin}; the
resulting formulation is particularly suited for applications in quantum optics and for
building up a photon detection theory \cite{Bar90QO,BarSpringer,Carm08,ZolG97,GarZ00}.
However, the theory is fully developed only for the case in which the HP-equation involves
only bounded operators in the initial Hilbert space. Many results are known on the existence
and uniqueness of the solution of the HP-equation with unbounded coefficients
\cite{Fagnola06,Fagnola99,FagW03}; our aim is to combine these results with the equations for
the (unbounded) field observables and to show how to arrive to the key evolution equation
\eqref{foreward} of the theory of continuous measurements, which concerns the ``reduced
characteristic operator'' \eqref{eq:defG}. Moreover, for applications, it is important to
consider the case in which the initial state of the quantum fields is not only the vacuum, but
at least a generic coherent vector. This gives that the reduced dynamics is not a quantum
dynamical semigroup and we need to handle a master equation with a time-dependent, unbounded
Liouville operator. Finally, a relevant physical example is given: the degenerate parametric
oscillator \cite{Carm08}. The paper is based on Castro's PhD thesis \cite{Castro}.

For any separable complex Hilbert space $\mathfrak{h}$ let us introduce the following classes
of operators on it: $\Lscr(\mathfrak{h})$, the space of bounded linear operators,
$\Uscr(\mathfrak{h})$ the class of the unitary operators, $\Tscr(\mathfrak{h})$ the
trace-class, $\Sscr(\mathfrak{h}):=\bigl\{\rho\in\Tscr(\mathfrak{h}):\rho\ge0$,
$\Tr\{\rho\}=1\bigr\}$ the set of statistical operators.

Then, we introduce the symmetric Fock space  $\mathscr{F}$ over $L^2(\mathbb{R}_+;\Zscr)$,
where $\Zscr$ is a $d$-dimensional complex Hilbert space (the \emph{multiplicity space}) in
which we fix a complete orthonormal system $\{z_i,\ i=1,\ldots,d\}$. We denote by $e(f)$ the
\emph{exponential vector} in the Fock space $\Fscr$ associated with the test function $f\in
L^2(\mathbb{R}_+;\Zscr)$ and we call \emph{coherent vector} $\psi(f):=\norm{e(f)}^{-1}e(f)$.
Recall that $\langle e(g)|e(f)\rangle = \exp \langle g|f\rangle$. We assume familiarity with
such notions and with quantum stochastic calculus \cite{Partha}. We shall use the notation
$f_i(t):=\la z_i|f(t)\ra$ for all $i\geq 1$ and we set $f_0(t)=1$. We fix the sets \[
\Mscr=L^2(\mathbb{R}_+;\Zscr)\cap L^\infty_{\mr{loc}}(\mathbb{R}_+;\Zscr) \qquad \text{and}
\qquad \Escr=\text{linear span of } \{e(f): f \in \Mscr\}.
\] The set $\Escr$ is dense in $\Fscr$ \cite[Corollary 19.5 p.\ 127]{Partha}.

An important feature of the Fock space $\Fscr$ is its structure of continuous tensor product.
For any choice of the times $0 \le s\le t$ let us introduce the symmetric Fock space $
\Fscr_{(s,t)}$ over ${L}^2 ((s,t);\Zscr)$ and the symmetric Fock space $\Fscr_{(t}$ over
${L}^2 ((t,\infty);\Zscr)$. Then, we have the natural identifications
\begin{equation}\label{eq:Fock-fac}
\Fscr\simeq \Fscr_{(0,s)}\otimes  \Fscr_{(s,t)}\otimes   \Fscr_{(t} \qquad \textrm{and}
\qquad e(f)\simeq e(f_{(0,s)})\otimes  e(f_{(s,t)})\otimes  e(f_{(t})\,,
\end{equation}
where $f_{(s,t)}(x):=1_{(s,t)}(x)f(x)$ and $ f_{(t}(x):=1_{(t,\infty)}(x)f(x)$. The symbol
$\otimes $ denotes the tensor product of Hilbert spaces, vectors and operators; the algebraic
tensor product of dense spaces is denoted by $\odot$.

The \emph{Weyl operator}\cite{Partha} $W(g;U)$, with $g\in\LII$ and $U\in
\Uscr\big(\LII\big)$, is the unique unitary operator on $\Fscr$ defined by
\begin{equation}\label{def:Weyl}
W(g;U)\,e(f)=\exp\Bigl\{-\frac{1}{2}\|g\|^2-\la
g|\ U f\ra\Bigr\}\,e(Uf+g),\qquad \forall f\in\LII .
\end{equation}
From the definition one obtains the relations \[ W(g;U)^{-1}=W(g;U)^* =W(-U^*g;U^*) \] and the
composition law
\begin{equation}\label{eq:Weyl-composition}
W(h;V)W(g;U)=\exp \big\{-\rmi\IM\la h|Vg\ra\big\}W(h+Vg;VU).
\end{equation}

Finally, we denote by $A_i^\dagger(t)$, $\Lambda_{ij}(t)$, $A_j(t)$ the creation, gauge and
annihilation processes associated with the basis $\{z_i,\ i=1,\ldots,d\}$; we shall use also
the notation
\begin{equation}\label{Lambdaij}
\Lambda_{i0}(t)=A^\dag_i(t), \quad \Lambda_{0j}(t)=A_j(t), \quad \Lambda_{00}(t)=
t, \qquad i,j=1,\ldots,d.
\end{equation}
In particular we have $\la e(g)|\Lambda_{ij}(t)e(f)\ra= \int_0^t \overline{g_i(s)}\,f_j(s)\rmd
s \,\la e(g)|e(f)\ra$, $i,j=0,\ldots,d$.

Let $\Hscr$ be a separable complex Hilbert space, the \emph{initial space}, and let us call
$S_{\Hscr}$ the quantum system with Hilbert space $\Hscr$.

We refer to \cite{Partha,Fagnola06} for the definition of quantum stochastic integrals with
respect to the operator noises $\Lambda_{ij}$, but we need to report at least the notions of
adapted process and stochastic integrability.

\begin{definition}[\!{\!\cite[p.\ 180]{Partha}, \cite[Definition 2.1]{Fagnola06}}] \label{defi:adap-proc}
Let ${D}$ be a dense manifold in $\Hscr$. A family $\{L(t), t \ge 0 \}$ of operators in
$\Hscr\otimes  \Fscr$ is an \emph{adapted process} with respect to $({D},\Mscr)$ if (i)
$D\odot\Escr\subset \bigcap_{t\ge0}\mr{Dom}(L(t))$, (ii) the map $t\mapsto L(t)u\otimes  e(f)$
is strongly measurable, $\forall u\in {D},\ f\in \Mscr$, (iii) $L(t)u\otimes  e(f_{(0,t)}) \in
\Hscr\otimes  \Fscr_{(0,t)}$ and $ L(t)u\otimes  e(f)= \big(L(t)u\otimes  e(f_{(0,t)})\big)
\otimes e(f_{(t})$, $\forall t \ge 0, u\in {D}, f\in\Mscr$.

If additionally the map $t\mapsto L(t)u\otimes e(f)$ is continuous for every $u\in {D}$ and
$f\in\Mscr$ the process is said to be \emph{regular adapted}. Moreover, the adapted process
$L$ is said to be \emph{stochastically integrable} if, for all $ t\ge 0$, $u\in {D}$ and
$f\in\Mscr$, one has $ \int_0^t\|L(s)u\otimes e(f)\|^2\rmd s < \infty$.
\end{definition}

A key notion in the construction of dilations of quantum dynamical semigroups is the one of
cocycle \cite{Acc78}. We introduce the strongly continuous one-parameter semigroup
$\{\theta(t),\,t\geq 0\}$ of the shift operators on $\mr{L}^2(\mathbb{R}_+;\Zscr)$ and its
second quantisation $\Theta$ on $\Fscr$: for every $t\geq 0$
\begin{equation}\label{eq:Theta}
(\theta_tf)(x)=f(x+t) \qquad \textrm{and} \qquad
\Theta_te(f)=e(\theta_tf), \qquad \forall f \in
\mr{L}^2(\mathbb{R}_+;\Zscr).
\end{equation}
Let us note that, for $r<s$, $\left(\theta_t1_{(r,s)}\right)(x)=1_{(r,s)}(x+t)=
1_{(r-t,s-t)}(x)$; this implies
\[
\Theta_te\big(f_{(0,s)}\big)=e(0), \quad \text{for } 0<s\leq t, \qquad \Theta_t \Fscr_{(r,s)}\subset
\Fscr_{(r-t,s-t)}\quad \text{for } 0\leq t \leq r <s.
\]
Moreover, it turns out that $\Theta_t^{\,*}$ is an isometry. We extend $\Theta_t$ to the space
$\Hscr\otimes \Fscr$ by stipulating that it acts as the identity on $\Hscr$.

\begin{definition}[Right and left cocycles]\label{defi:cocycle}
A bounded, adapted operator process $X(t)$ in $\Hscr\otimes \Fscr$ is called \emph{right
cocycle} (respectively,  \emph{left cocycle}) if for every $s,t \ge 0$ we have
$X(t+s)=\Theta^{\,*}_sX(t)\Theta_sX(s)$ $\big(X(t+s)=X(s)\Theta^{\,*}_sX(t)\Theta_s\big)$.
\end{definition}

\section{The Hudson-Parthasarathy equations}

Let us consider the quantum stochastic differential equation (QSDE) for operators on
$\Hscr\otimes \Fscr$, known as \emph{right} HP-equation: $U(0)=\1$,
\begin{equation} \label{eq:right}
\rmd U(t)=\displaystyle \Big(\sum_{i\ge1} R_i \rmd A_i^{\dag}(t)
+ \sum_{i,j\ge1} F_{ij}\rmd\Lambda_{ij}(t)+\sum_{j\ge1} N_j
\rmd A_j(t) + K\rmd t \Big)U(t),
\end{equation}
where the coefficients $K,\ R_i,\ N_i,\ F_{ij}$, with $i,j= 1,\ldots,d$, are (possibly
unbounded) operators in the initial space $\Hscr$. Very general sufficient conditions, which
guarantee the existence of a unique solution of \eqref{eq:right} and the fact that such a
solution is a \emph{unitary cocycle}, are given by Fagnola and Wills \cite{FagW03}.

By using the notation \eqref{Lambdaij} and by setting
\begin{equation}
F_{00}=K\,, \quad F_{i0}=R_i\,, \quad F_{0j}=N_j\,,
\end{equation}
we  can write the right HP-equation in the shortened form
\begin{equation}\label{eq:rHP}
\rmd U(t)= \sum_{i,j\ge 0}F_{ij}\, \rmd\Lambda_{ij}(t)\, U(t),\qquad
U(0)=\1.
\end{equation}
We shall need also the adjoint equation, the \emph{left HP-equation}:
\begin{equation}\label{eq:lHP}
\rmd V(t)= V(t) \sum_{i,j\ge 0}F_{ji}^{\;*}\, \rmd\Lambda_{ij}(t),\qquad
V(0)=\1.
\end{equation}

\begin{definition}[{Right Solution -- \cite[Definition 3.2]{Fagnola06}}]\label{defi:R-S}
Let ${D}$ be a dense subspace in $\Hscr$. An operator process $U$ is a \emph{solution of the
right HP-equation} in ${D} \odot  \Escr$ for the matrix $F$ if:
\begin{description}%[(ii)]
\item[(i)] each operator $F_{ij}\otimes \1$ is closable and $\displaystyle \bigcup_{t\ge0}
    U(t)({D}\odot  \Escr) \subset \bigcap_{i,j\ge0}\textrm{Dom}(\overline{F_{ij}\otimes
    \1})$;

\item[(ii)] each process $\overline{F_{ij}\otimes \1}\ U$ is stochastically integrable and
    \[ U(t)=\1+\sum_{i,j\ge 0}\int_0^t \overline{F_{ij}\otimes \1}\;
    U(s)\,\rmd\Lambda_{ij}(s) \qquad \text{on } D \odot \Escr, \quad \forall t\ge 0.\]
\end{description}
\end{definition}

\begin{definition}[{Left Solution -- \cite[Definition 3.1]{Fagnola06}}]\label{defi:L-S}
Let ${\til{D}}$ be a dense subspace in $\Hscr$. An operator process $V$ is a \emph{solution of
the left HP-equation} in $\til{{D}} \odot \Escr$ for the matrix $F^*$ if:
\begin{description}%[(ii)]
\item[(i)] $ {\til{D}} \subset \bigcap_{i,j\geq 0}\textrm{Dom}(F_{ij}^{\;*})$ and the
    linear manifold $\left(\bigcup_{i,j\geq 0}
    F_{ij}^{\;*}\left({\til{D}}\right)\right)\odot \mathcal{E}$ is contained in the domain
    of $V(t)$, $\forall t \geq 0$;

\item[(ii)] the processes $\big(V(t)F_{ij}^{\;*};t\ge 0\big)$ are stochastically
    integrable and
\[ V(t)=\1+\sum_{i,j\ge 0}\int_0^tV(s)F_{ji}^{\;*}\,
    \rmd\Lambda_{ij}(s) \qquad \text{on } \til{D} \odot \Escr, \quad \forall t\ge
    0.\]
\end{description}
\end{definition}

\begin{hypo}\label{hyp:C} (The matrix $F$)

\begin{description}
\item[(i)] $F=(F_{ij};0\leq i,j\leq d)$ is a matrix of closed operators in the initial
    space $\Hscr$.  By $F^*$ we denote the adjoint matrix, defined by
    $(F^*)_{ij}=F^*_{ji}$. We also define \ $\Dom(F):= \bigcap_{i,j\geq
    0}\textrm{Dom}(F_{ij})$, \quad$\Dom(F^*):= \bigcap_{i,j\geq 0}\textrm{Dom}(F_{ji}^*)$
\item[(ii)] For $1\leq i,j\leq d$, we have $F_{ij}=S_{ij}- \delta_{ij}\1$, where the
    $S_{ij}$ are bounded operators on $\Hscr$ satisfying the unitarity conditions \quad$
    \sum_{k=1}^d S^*_{ki}S_{kj}=\sum_{k=1}^d S_{ik}S^*_{jk} = \delta_{ij} $.
\item[(iii)] There exist a dense subspace $D$ which is a core for $K,\, R_i,\, N_i$,
    $i=1,\ldots,d$, and a dense subspace $\til D$ which is a core for $K^*,\, R_i^*,\,
    N_i^*$, $i=1,\ldots,d$.
\item[(iv)] $\mathrm{Dom}(N_i^*)\supset D\cup \til D$, \ $\mathrm{Dom}(R_i) \supset D\cup
    \til D$, \ $\mathrm{Dom}(N_i)\supset\mathrm{Dom}(K)$, \ $ \forall i\geq 1$.
\item[(v)] $\forall k\ge1$, $\forall u\in \mathrm{Dom}(K)$:
    $S_{ki}u\in\mathrm{Dom}(R^*_k)$, $\forall i\ge1$.
\item[(vi)] The operators $K$ and $K^*$ are the infinitesimal generators of two strongly
    continuous contraction semigroups on $\Hscr$. Moreover, we have $\forall u\in {D}$,
    $\forall v\in\til{{D}}$
\begin{equation}\label{dissipativity}
2\mr{Re}\la Ku|u\ra=-\sum_{k\ge1}\|R_ku\|^2 , \qquad
2\mr{Re}\la K^*v|v\ra=-\sum_{k\ge1}\|N^*_{k}v\|^2.
\end{equation}
\item[(vii)] $ N_i^*u=- \sum_{k\ge1} S_{ki}^*R_ku\,, \qquad  \forall u \in D\cup \til D\,,
    \quad \forall i\geq 1$.

\item[(viii)] There exist a positive self-adjoint operator $C$ on $\Hscr$ and the
    constants $\delta> 0$ and $b_1,\,b_2 \ge 0$ such that \cite[pp.\ 281--291]{FagW03}
    $\mr{Dom}(C^{1/2})\subset \mr{Dom}(F)$ and
\begin{description}
\item[(a)] for each $\epsilon\in (0,\delta)$, there exists a dense subspace ${D}_\epsilon
    \subset\til{D}$ such that $C_\epsilon^{1/2} {D}_\epsilon\subset\til{D}$ and each
    operator $F_{ij}^*C_\epsilon^ {1/2}|_{{D}_{\epsilon}}$ is bounded, where
    $C_\epsilon=\frac{C}{(1+\epsilon C)^2}$;
\item[(b)] for all $0<\epsilon<\delta$ and $u_0,\ldots,u_d\in \mr{Dom}(F)$, the following
    inequality holds:
\begin{multline*}
\sum_{i,j\geq0}\left( \la u_i|C_\epsilon F_{ij}u_j\ra +\la F_{ji}u_i|C_\epsilon u_j\ra +
\sum_{k\geq 1}\la F_{ki}u_i|C_\epsilon F_{kj}u_j\ra\right)
\\ {}\leq  \sum_{i\geq 0}\left(b_1
\la u_i| C_\epsilon u_i\ra+ b_2 \norm{u_i}^2\right).
\end{multline*}
\end{description}
\end{description}
\end{hypo}

\begin{proposition}\label{prop:moreonF}
Under Hypothesis \ref{hyp:C} also the following properties hold:
\begin{enumerate}
\item $\mathrm{Dom}(R_k)\supset \mathrm{Dom}(K)\cup \mathrm{Dom}(K^*)$,
    $\mathrm{Dom}(N^*_k)\supset \mathrm{Dom}(K)\cup\mathrm{Dom}(K^*)$,  $k\geq 1$.

\item $\mathrm{Dom}(F)=\mathrm{Dom}(K)$; Eqs.\ \eqref{dissipativity} hold $\forall u\in
    \mathrm{Dom}(K)$, $\forall v\in \mathrm{Dom}(K^*)$.

\item $N_i^*u=- \sum_{k\ge1} S_{ki}^*R_ku$, $R_iu=-\sum_{k\ge1} S_{ik}N_k^*u$, $\forall u
    \in \mathrm{Dom}(K)\cup \mathrm{Dom}(K^*)$, $\forall i\geq 1$.

\item $ N_iu=-\sum_{k\ge1}R^*_kS_{ki}u \,, \qquad \forall u\in \mathrm{Dom}(K), \quad
    \forall i=1,\dots,d $.

\item for every choice of $u_0,\,u_1,\ldots,\,u_d$ in $\mathrm{Dom}(F)$ and of
    $v_0,\,v_1,\ldots,\,v_d$ in $\til D$, we have
\begin{subequations}\label{unitconds}
\begin{gather}\label{thetaF=0}
\sum_{i,j\geq 0}\left( \la u_i|F_{ij}u_j\ra + \la F_{ji}u_i|u_j\ra + \sum_{k\geq 1} \la
F_{ki}u_i|F_{kj}u_j\ra\right) =0.
\\
\sum_{i,j\geq 0}\left( \la v_i|F_{ji}^{\;*}v_j\ra + \la F_{ij}^{\;*}v_i|v_j\ra +
\sum_{k\geq 1} \la F_{ik}^{\;*}v_i|F_{jk}^{\;*}v_j\ra\right) =0.
\end{gather}
\end{subequations}
\end{enumerate}
\end{proposition}

\begin{proof}
By \eqref{dissipativity} we get, $\forall \phi\in D$, $\|R_k\phi\|^2 \leq 2|\la
K\phi|\phi\ra|$. For any $u\in \mathrm{Dom}(K)$ we can find a sequence $u_n\in D$ converging
to $u$. Then, $Ku_n\to Ku$ weakly and by the proposition at p.\ 112 of [\citen{ReedS}] the
sequence $Ku_n$ is norm bonded: $\|Ku_n\| \leq c$. Then, we have $ \|R_k(u_n-u_m)\|^2 \leq 2
|\la K(u_n-u_m)|u_n-u_m\ra| \leq 4 c \|u_n-u_m\|$. Therefore, $R_ku_n$ is a Cauchy sequence.
Being $R_k$ closed, $\left(u, \lim_n R_k u_n\right)$ belongs to the graph of $R_k$ and, so,
$u\in \mathrm{Dom}(R_k)$; this gives $\mr{Dom}(R_k)\supset \mr{Dom}(K)$. By the same property
of closed operators, we get that the first equation in \eqref{dissipativity} can be extended
to the whole $\mr{Dom}(K)$. Similarly, from the second equation in \eqref{dissipativity} we
get that it can be extended to the whole $\mr{Dom}(K^*)$ and that $\mr{Dom}(N^*_k)\supset
\mr{Dom}(K^*)$. Once again by the property above of closed operators and by the unitarity of
the operator matrix $S$, we get that point (vii) can be extended to
$\mr{Dom}(K)\cup\mr{Dom}(K^*)$. Therefore, on the same domain, $\sum_{k\ge1}\|R_{k}u\|^2=
\sum_{k\ge1}\|N^*_{k}u\|^2$. By exchanging $R_k$ and $N_k^*$ in the two equations in
\eqref{dissipativity} we get also Dom$(R_k)\supset \mathrm{Dom}(K^*)$ and Dom$(N^*_k)\supset
\mathrm{Dom}(K)$. By these results and the definition of Dom$(F)$ we get
$\mathrm{Dom}(F)=\mathrm{Dom}(K)$. This ends the proof of points (1)--(3).

By using the first equation in point (3) of this proposition and point (v) in Hypothesis
\ref{hyp:C} we have, $\forall u\in\mr{Dom}(K)$, $\forall v\in\mr{Dom}(K)\cup\mr{Dom}(K^*)$,
\[
\la v|N_i u\ra=\la N^*_iv| u\ra =-\sum_{k\ge1}\big\la S^*_{ki}R_k v\big|u\big\ra =
-\sum_{k\ge1}\big\la R_k v\big|S_{ki}u\big\ra=-\sum_{k\ge1}\big\la  v\big|R^*_kS_{ki}u\big\ra.
\]
By the density of $\mr{Dom}(K)\cup\mr{Dom}(K^*)$ in $\Hscr$ we have the statement in point
(4). Equations in point (6) are by direct verification.
\end{proof}

Note the difference in the domains of Eqs.\ \eqref{unitconds}. The last requirement in point
(iv) has been added just to have Eq.\ \eqref{thetaF=0} on the whole $\mathrm{Dom}(F)$ and not
only on $D$. As one can check, also Hypothesis \textbf{HGC} at p.\ 205 of [\citen{Fagnola06}]
holds for the coefficients of the left HP-equation \eqref{eq:lHP}. We collect in the following
theorem many results.

\begin{theorem}[\!{\!\cite[Proposition 2.2, Theorem 2.3]{FagW03}; \cite[Proposition 6.3,
Theorems 8.4, 8.5]{Fagnola06}}] \label{teo:SolutionHPequation} Under Hypothesis \ref{hyp:C}
the left HP-equation \eqref{eq:lHP} has a unique solution $\{V(t);\, t\geq 0\}$ on $\til
D\odot \Escr$, which is a strongly continuous left cocycle of contractions. Moreover,
$U(t)=V(t)^*$ is a right cocycle and solves the right HP-equation on
$\mathrm{Dom}(C^{1/2})\odot \Escr$.
\end{theorem}

As we take $U=V^*$, if $V$ is an isometry process, $U$ is a coisometry process and vice versa.
The following Proposition is a small variation of Corollary 2.4 of [\citen{FagW03}] or of
Corollary 11.2 of [\citen{Fagnola06}].

\begin{proposition}\label{teo:U-isometry}
Under Hypothesis \ref{hyp:C} the contractive solution $U$ of \eqref{eq:rHP} introduced in
Theorem \ref{teo:SolutionHPequation} is a strongly continuous isometric process. Moreover, if
$U$ is unitary, it is the unique bounded solution on $\mathrm{Dom}(C^{1/2})\odot\Escr$ of
\eqref{eq:rHP}.
\end{proposition}
\begin{proof}
Let $\til U$ be another bounded solution and apply the second fundamental formula of QSC
(\!\!\cite[Proposition 25.2]{Partha}, \cite[Eq.\ (2)]{Fagnola06}) to $U,\, \til U$. We get,
$\forall f,g \in \Mscr$, $\forall u,v\in \mathrm{Dom}(C^{1/2})$,
\begin{multline*}
\big\la \til U(t)v\otimes e(g)\big| U(t)u\otimes e(f)\big\ra - \big\la v\otimes e(g)\big|
u\otimes e(f)\big\ra
\\
{} = \sum_{i,j\ge 0}\int_0^t\rmd s\ \overline{g_i(s)}\Big\{\big\la \til U(s)v\otimes e(g)\big|
F_{ij}\otimes \1\, U(s)u\otimes e(f)\big\ra
\\
{}+ \big\la F_{ji}\otimes \1\,\widetilde{U}(s)v\otimes e(g)\big| U(s)u\otimes e(f)\big\ra
\\
{}+ \sum_{k\ge1}\big\la F_{ki}\otimes \1  \,\til U(s)v\otimes e(g)\big| F_{kj}\otimes \1
\,U(s)u\otimes e(f)\big\ra\Big\}f_j(s)\,.
\end{multline*}
But $U$ and $\til U$ are solutions on $\mathrm{Dom}(C^{1/2})\odot\Escr$; by point (i) of
Definition \ref{defi:R-S} and Eq.\ \eqref{thetaF=0}, we have that the integrand vanishes.
Therefore, $\forall f,g \in \Mscr$, $\forall u,v\in \mathrm{Dom}(C^{1/2})$, $ \big\la \til
U(t)v\otimes e(g)\big| U(t)u\otimes e(f)\big\ra = \big\la v\otimes e(g)\big| u\otimes
e(f)\big\ra$ and this is equivalent to $\til U(t)^* U(t)=\1$. This equation for $\til U=U$
gives that $U(t)$ is isometric, while for $U$ unitary gives $\til U(t)^*= U(t)^*$. Being the
adjoint of a strongly continuous process, $U$ is weakly continuous and, being an isometry, it
is strongly continuous.
\end{proof}

\section{Quantum dynamical semigroups and unitary cocycles}

For the physical interpretation of the evolution operator $U(t)$ we need it to be a strongly
continuous cocycle of unitary operators (see [\citen{BarSpringer}] Section 2.2 and references
there in). The unitarity is associated to some property of a related quantum dynamical
semigroup (QDS); so, we start with some notions on QDSs.

\begin{definition}[QDS]\label{defi:QDS}
Let us consider a family $\left\{\Tcal(t),\,t\geq 0\right\}$ of bounded operators on
$\Lscr(\Hscr)$ with the following properties:
\begin{description}
\item[(i)] $\Tcal(t+s)=\Tcal(t)\Tcal(s)$, $\forall s,t \ge0$, and $\Tcal(0)$ is the
    identity map;
\item[(ii)] $\Tcal(t)$ is completely positive, $\forall t \ge 0$;
\item[(iii)] $\Tcal(t)$ is a $\sigma$-weakly continuous operator on $\Lscr(\Hscr)$,
    $\forall t\ge0$;
\item[(iv)] for each $X\in \Lscr(\Hscr)$ the map $t \mapsto \Tcal(t)[X]$ is continuous
    with respect to the $\sigma$-weak topology of $\Lscr(\Hscr)$.
\end{description}
Then, the family of operators $\Tcal(t)$ is called a \emph{quantum dynamical semigroup}. If
also $\Tcal(t)[\1]=\1$ holds $\forall t\geq 0$, the QDS $\Tcal(t)$ is said to be \emph{Markov}
or \emph{conservative}.
\end{definition}

\begin{theorem}[\!{\!\cite[Theorem 3.22, Corollary 3.23]{Fagnola99}}]\label{teo:QDS-min}
Let $A$ be the infinitesimal generator of a strongly continuous contraction semigroup in
$\Hscr$ and let $L_k$, $k=1,\ldots$, be operators in $\Hscr$ such that the domain of each
operator $L_k$ contains the domain of $A$ and  for every $u \in \mathrm{Dom}(A)$ we have $2\RE
\la u|Au\ra+\sum_{k\ge1}\norm{L_ku}^2 = 0$.

For all $X\in \Lscr(\Hscr)$, let us consider the quadratic form $\Lcal[X]$ in $\Hscr$ with
domain $\mr{Dom}(A)\times\mr{Dom}(A)$ given by
\begin{equation}\label{eq:cuadratic-form}
\la v|\Lcal[X]u\ra=\la v|XAu\ra+\la Av|Xu\ra+\sum_{k\ge1}\la L_kv|XL_ku\ra.
\end{equation}

Then, there exists a QDS $\Tcal(t)$ solving the equation
\begin{equation}\label{eq:cuadratic-equation}
\la v|\Tcal(t)[X]u\ra=\la v|Xu\ra+\int^t_0\big\la v\big|\Lcal\big[\Tcal(s)[X]\big]u\big\ra\rmd s
\end{equation}
with the property that $\Tcal(t)[\1]\le \1$,  $\forall t \ge0$, and such that for every
$\sigma$-weakly continuous family $\Tcal^\prime(t)$ of positive maps on $\Lscr(\Hscr)$
satisfying Eqs.\ \eqref{eq:cuadratic-form} and \eqref{eq:cuadratic-equation} we have
$\Tcal(t)[X]\le\Tcal^\prime(t)[X]$, $\forall t\ge0$, for all positive $X \in \Lscr(\Hscr)$.

If moreover the QDS $\Tcal(t)$ is conservative, then it is the unique $\sigma$-weakly
continuous family of positive maps on $\Lscr(\Hscr)$ satisfying Eq.\
(\ref{eq:cuadratic-equation}).
\end{theorem}

The QDS $\Tcal(t)$ defined in Theorem \ref{teo:QDS-min} is called the \emph{minimal quantum
dynamical semigroup} generated by $A$ and $L_k$, $k=1,\ldots$. Sufficient conditions to assure
Markovianity of a QDS are known \cite{Fagnola99,Fagnola06}. In the application we shall use
the following result.

\begin{theorem}[\!{\cite[Theorem 9.6]{Fagnola06}}] \label{theor_QZ}
Let $A$, $L_k$ be as in Theorem
\ref{teo:QDS-min} and suppose that there exist two positive self-adjoint operators $Q$ and $Z$
in $\mathscr{H}$ with the following properties:
\begin{itemize}
\item $\Dom(A)$ is contained in $\Dom(Q^{\frac{1}{2}})$ and is a core for
    $Q^{\frac{1}{2}}$;
\item the linear manifold $\bigcap_{k\ge1}L_k\big(\Dom(A^2)\big)$ is contained in
    $\Dom(Q^{\frac{1}{2}})$;
\item $\Dom(A)\subset\Dom(Z^{\frac{1}{2}})$  and
\[
 -2\mathrm{Re}\la u|Au\ra=\sum_{k\ge1}\|L_ku\|^2=\|Z^{\frac{1}{2}}u\|^2,\qquad  \forall
u\in\Dom(A);
\]
\item $\Dom(Q)\subset\Dom(Z)$ and for all $u\in\Dom(Q^{\frac{1}{2}})$ we have
    $\|Z^{\frac{1}{2}}u\|\le\|Q^{\frac{1}{2}}u\|$;
\item there is a positive constant $b$ depending only on $A$, $L_k$, $Q$ such that, for
    all $u\in\Dom(A^2)$, the following inequality holds
\[
2\RE\la Q^{\frac{1}{2}}u\big|Q^{\frac{1}{2}}Au\ra
+\sum_{k\geq 1}\|Q^{\frac{1}{2}}L_ku\|^2 \le b\|Q^{\frac{1}{2}}u\|^2.
\]
\end{itemize}
Then, the minimal quantum dynamical semigroup associated to $A$ and $L_k$ is Markov.
\end{theorem}

Now, we can go back to the problem of the unitarity of $U(t)$.

\begin{theorem}[\!{\!\cite[Theorems 10.2, 10.3]{Fagnola06}}] \label{teo:V-isometry}
Under Hypothesis \ref{hyp:C} the contractive left cocycle $V$ solving \eqref{eq:lHP} is such
that the family of operators $\til{\Tcal}(t)$ defined by
\[
\la v|\til{\Tcal}(t)[X]u\ra = \la V(t)v \otimes e(0)|(X\otimes \1)V(t)u\otimes e(0)\ra , \quad \forall
u,v\in \Hscr, \quad \forall X\in\Lscr(\Hscr).
\]
is the minimal QDS generated by $K^*$ and $N_{k}^*$, $k=1,\ldots,d$.

Moreover, the following conditions are equivalent:
\begin{description}
\item[(ii)] the process $V$ is an isometry;
\item[(ii)] the minimal QDS associated with $K^*$ and $N^*_{k}$ is conservative;
\end{description}
\end{theorem}

\begin{hypo}[Markov condition] \label{hyp:Markov}
The minimal QDS generated by $K^*$ and $N_{k}^*$, $k=1,\ldots,d$, is conservative.
\end{hypo}

Now, both $U(t)$ and $V(t)=U(t)^*$ are isometries and, so, they are unitary operators.

\begin{corollary}\label{cor:unitary}
Under Hypotheses \ref{hyp:C} and \ref{hyp:Markov} the process $U$ introduced in Theorem
\ref{teo:SolutionHPequation} is unitary and it is the unique bounded solution on
$\mathrm{Dom}(C^{1/2})\odot\Escr$ of \eqref{eq:rHP}.
\end{corollary}

\begin{remark}\label{Uts}
When $U(t)$ is a strongly continuous unitary right cocycle, one can define the unitary
evolution
\begin{equation}\label{eq:U(t,s)}
U(t,s):=U(t)U(s)^* , \qquad t\geq s\geq 0.
\end{equation}
It is easy to check that it is strongly continuous in $t$ and $s$ and such that
\begin{equation}\label{eq:U-comp}
U(t,s)=\Theta^*_s U(t-s)\Theta_s\,, \qquad
U(t,r)=U(t,s)U(s,r),\qquad 0\leq r\le s\le t\,.
\end{equation}
Moreover, the operator $U(t,s)$ is adapted to $\Hscr\otimes \Fscr_{(s,t)}$ in the sense that
it acts as the identity on $\Fscr_{(0,s)}\otimes \Fscr_{(t}$ and leaves $\Hscr\otimes
\Fscr_{(s,t)}$ invariant. The unitary operator $U(t,s)$ is interpreted as the evolution
operator of system $S_{\Hscr}$ and fields (described in the Fock space $\Fscr$) in the
interaction picture with respect to the free dynamics of the fields \cite{Fri85,BarSpringer}.
\end{remark}

\section{Observables and instruments}
Let us consider now the case in which we are interested in the behaviour of the system
$S_\Hscr$, but any action on it is mediated by some input/output fields, for instance the
electromagnetic field, represented in the Fock space $\Fscr$. In this situation we can measure
only field observables, from which we make inferences on $S_\Hscr$; we can speak of an
indirect measurement. By choosing field observables which commute also at different times in
the Heisenberg picture, we can represent also measurements in continuous time. To this end we
show how to construct such observables and how to eliminate the fields by a partial trace; in
this way we obtain a description of the continuous measurement in terms of quantities
(instruments) related to the system $S_\Hscr$ alone \cite{BarSpringer,BarL85JMP}.

To give the observables at all times, we have to give infinitely many commuting selfadjoint
operators or their joint spectral measure; the easiest way to do this is to work with the
``Fourier transform'' of such a spectral measure, the \emph{characteristic operator}. The
construction below involves the Weyl operators \eqref{def:Weyl}.

\begin{hypo}[Elements of the characteristic operator] \label{H:theBs}
Let $B^1,B^2,\ldots,B^m$ be commuting selfadjoint operators on $\Zscr$ and let us take $c\in
L^1_{\mathrm{loc}}(\Real_+; \Real^m)$, $b \in L^2_{\mathrm{loc}}(\Real_+; \Zscr)$ and
$h^\alpha \in L^2_{\mathrm{loc}}(\Real_+; \Zscr)$, $\alpha = 1,\ldots, m$, such that
\begin{equation}\label{proph}
\IM \langle h^\alpha(t)|h^\beta(t)\rangle = 0\,, \qquad B^\alpha\, h^\beta(t)=0\,,
\qquad \forall t\geq 0, \quad  \forall \alpha, \beta = 1,\ldots, m.
\end{equation}
\end{hypo}

\begin{definition}[The characteristic operator]\label{def:Phi}
For any \emph{test function} $k\equiv (k_1,\ldots$, $k_m) \in L^\infty(\Real_+; \Real^m)$ let
us define $\mathsf{S}_t(k) \in \Uscr\big(L^2(\Real_+;\Zscr)\big)$, $\mathsf{r}_{t}(k)\in
L^2(\Real_+;\Zscr)$ and the \emph{characteristic operator} $\widehat \Phi_t(k)\in
\Uscr(\Fscr)$ by
\begin{gather}\nonumber
\big(\mathsf{S}_t(k) f\big)(s) = 1_{(0,t)}(s) \left[ \mathsf{S}\big(k(s)\big)- \1  \right]
f(s) +f(s)\,, \quad \forall f\in L^2(\Real_+;\Zscr)\,,
\\ \label{subeq:Sr}
\mathsf{S}\big(k(s)\big)= \prod_{\alpha=1}^m \rme^{\rmi k_\alpha(s) B^\alpha}\,, \qquad
\mathsf{r}_{t}(k)(s)= 1_{(0,t)}(s)\, \mathsf{r}(k;s)\,,
\\ \nonumber
\mathsf{r}(k;s)= \rmi \sum_{\alpha=1}^m k_\alpha(s) h^\alpha(s) +
\left[\mathsf{S}\big(k(s)\big) -\1 \right]b(s)\,,
\end{gather}
\begin{multline}\label{def:hatphi}
\widehat \Phi_t(k) = \exp\left\{ \rmi\int_0^t \rmd s \left[ \sum_{\alpha=1}^m k_\alpha(s)
c^\alpha(s) + \IM \left\langle b(s) \big|\mathsf{S}\big(k(s)\big) b(s)\right\rangle\right]
\right\}
\\ {}\times W
\big(\mathsf{r}_{t}(k); \mathsf{S}_t(k)\big).
\end{multline}
\end{definition}

By \eqref{proph} and \eqref{subeq:Sr}, we get $\mathsf{S}(-k)= \mathsf{S}(k)^*$ and
$\mathsf{S}\big(k(t)\big)^* \mathsf{r}(k;t)=-\mathsf{r}(-k;t)$.

\begin{theorem}[\!{\!\cite[Theorem 3.1]{BarSpringer}}] \label{generalobservables}
Under Hypothesis \ref{H:theBs}, the characteristic operator introduced in Definition
\ref{def:Phi} has the following properties:
\begin{enumerate}
\item \emph{localisation properties:} $ \widehat \Phi_t\big( 1_{(t_1,t_2)}k\big) =
    \widehat \Phi_{t_2}\big( 1_{(t_1,t_2)}k\big)\in \Uscr\left(\Fscr_{(t_1,t_2)}\right)$,
    \ $ 0\leq t_1<t_2 \leq t$;

\item \emph{group property:}  $\widehat \Phi_t(0)=\mathds{1}$, \ $ \widehat
    \Phi_t(k)\,\widehat \Phi_t(k^\prime)=\widehat \Phi_t(k+k^\prime)$, \ $\forall
    k,k^\prime \in L^\infty(\Real_+; \Real^m)$;

\item \emph{continuity:} $\widehat \Phi_t(\kappa k)$ is strongly continuous in $\kappa \in
    \Real$ and in $t\geq 0$;
\item \emph{matrix elements:}
\begin{multline}\label{Phimatrixelements}
\langle e(g) | \widehat \Phi_t(k) e(f)\rangle = \langle e(g) |  e(f)\rangle
\\
{} \times \exp \biggl\{ \int_0^t \rmd s \biggl[ - \frac 1 2 \sum_{\alpha\beta} k_\alpha(s)
\langle h^\alpha(s) | h^\beta(s) \rangle k_\beta(s)
\\
{}+\rmi\sum_{\alpha=1}^m k_\alpha(s) \bigl( c^\alpha(s) + \langle h^\alpha(s)|f(s)\rangle
+ \langle g(s)| h^\alpha(s)\rangle \bigr)
\\
{}+ \left\langle g(s) +b(s) \big| \left(\mathsf{S}\big(k(s)\big)-\1  \right) \bigl( f(s) +
b(s) \bigr) \right\rangle \biggr] \biggr\};
\end{multline}
\item $\widehat \Phi_t(k)$ is the unique unitary solution of the QSDE
\begin{subequations}\label{PhiQSDE}
\begin{gather}
\rmd \widehat \Phi_t(k) = \sum_{i,j\geq 0}G_{ij}(t;k)\widehat \Phi_t(k)\, \rmd
\Lambda_{ij}(t), \qquad \widehat \Phi_0(k)=\1 ,
\\
\begin{split}
G_{00}(t;k) =&\langle b(t)| \left(\mathsf{S}\big(k(t)\big)-\1  \right)b(t)\rangle
\\ {}&+\rmi \sum_{\alpha=1}^m k_\alpha(t) c^\alpha(t) - \frac 1 2 \sum_{\alpha,\beta=1}^m
k_\alpha(t) \langle h^\alpha(t) | h^\beta(t) \rangle k_\beta(t),
\end{split}
\\
G_{j0}(t;k)=\langle z_j |\mathsf{r}(k;t)\rangle, \qquad G_{0j}(t;k)=-\langle
\mathsf{r}(k;t) | \mathsf{S}\big(k(t)\big)z_j\rangle,
\\
G_{ij}(t;k)=\left\langle z_i \big| \left(\mathsf{S}\big(k(t)\big)-\1  \right)z_j
\right\rangle .
\end{gather}
\end{subequations}
\end{enumerate}

Moreover, there exist a measurable space $(\Omega, \Dscr)$, a projection valued measure $\xi$
on $(\Omega, \Dscr)$, a family of real valued measurable functions $\big\{ \tilde
X(\alpha,t;\cdot)$, $\alpha=1,\ldots,m$, $t\geq 0\big\}$ on $\Omega$, a family of commuting
and adapted selfadjoint operators $\big\{ X(\alpha,t)$, $\alpha=1,\ldots,m$, $t\geq 0\big\}$
such that $\tilde X(\alpha,0;\omega)=0$, $X(\alpha,0)=0$ and, for any choice of $n$,
$0=t_0<t_1<\cdots < t_n \leq t$, $\kappa_\alpha^l\in \Real$,
\begin{multline}\label{Phiincrements}
\widehat \Phi_t(k) = \exp\biggl\{ \rmi \sum_{l=1}^n \sum_{\alpha=1}^m \kappa_\alpha^l \big[
X(\alpha,t_l) - X(\alpha,t_{l-1})\big]\biggr\}
\\
{}= \int_\Omega \exp\biggl\{ \rmi \sum_{l=1}^n \sum_{\alpha=1}^m \kappa_\alpha^l \big[ \tilde
X(\alpha,t_l;\omega) - \tilde X(\alpha,t_{l-1};\omega)\big]\biggr\} \xi(\rmd \omega)\,,
\end{multline}
where $k_\alpha(s)= \sum_{l=1}^n 1_{(t_{l-1},t_l)}(s) \,\kappa_\alpha^l$.
\end{theorem}

From the unitarity and the group property we have $\widehat \Phi_t(k)^*=\widehat
\Phi_t(k)^{-1}=\widehat \Phi_t(-k)$. Equation \eqref{PhiQSDE} is a right HP-equation with
trivial initial space; it can be written also in left form, with the same coefficients. Note
that $\overline{G_{ij}(t;-k)}= G_{ji}(t;k)$.

The observables $X(\alpha,t)$ can be identified by taking $k_\alpha(s)=\kappa$, $k_\beta(s)=0$
for $\beta\neq \alpha$; then, the first equality in \eqref{Phiincrements} gives $\rme^{\rmi
\kappa X(\alpha,t)}=\widehat \Phi_t(k)$ and, by differentiation of the matrix elements
\eqref{Phimatrixelements}, we get
\begin{multline*}
\langle e(g)|X(\alpha,t)e(f)\rangle =\langle e(g)|e(f)\rangle \int_0^t\,\rmd s \big\{
c^\alpha(s) \\ {}+ \langle h^\alpha (s)|f(s)\rangle +\langle g(s)|h^\alpha(s)\rangle + \langle
g(s)+b(s)|B^\alpha[f(s)+b(s)]\rangle\big\}.
\end{multline*}
Let us choose the complete orthonormal system $\{z_i,\, i=1,\ldots ,d\}$ in $\Zscr$ such that
it diagonalises all the operators $B^1,\ldots, B^m$ and such that its first $d^\prime $
components, $0\leq d^\prime\leq d$, span the intersection of the null spaces of these
operators; then, we have $B^\alpha=\sum_{i=d^\prime+1}^d B^\alpha_i |z_i\rangle \langle z_i|$,
$B^\alpha_i\in \Real$, and we can write, on the exponential domain,
\begin{multline}\label{observables}
X(\alpha,t) = \int_0^t c^\alpha(s)\,\rmd s +\sum_{i=1}^{d'}\int_0^t\left(\overline{ h^\alpha_i
(s)} \, \rmd A_i(s)+h^\alpha_i(s)\,\rmd A^\dagger_i(s)\right) \\{}+ \sum_{i=d'+1}^d
B^\alpha_i\int_0^t\left(\rmd \Lambda_{ii}+\overline{ b_i (s)} \, \rmd A_i(s)+b_i(s)\,\rmd
A^\dagger_i(s)+ \abs{b_i(s)}^2\rmd s\right).
\end{multline}
In quantum optical systems the continuous measurement of observables of the type
$\int_0^t\left(\overline{ h^\alpha_i (s)} \, \rmd A_i(s)+h^\alpha_i(s)\,\rmd
A^\dagger_i(s)\right)$ can be obtained by heterodyne/homodyne detection, while terms like
$\int_0^t\left(\rmd \Lambda_{ii}+\overline{ b_i (s)} \, \rmd A_i(s)+b_i(s)\,\rmd
A^\dagger_i(s)+ \abs{b_i(s)}^2\rmd s\right)$ are realised by direct detection, eventually
after interference with a known signal if $b_i\neq 0$\cite{Bar90QO}.

A key point in the whole construction is that even in the Heisenberg picture the observables
$X(\alpha,t)$ continue to be represented by commuting operators and, so, they can be jointly
measured also at different times. Let $U$ be a right unitary cocycle representing the
system-field dynamics and define $\forall T\geq 0$ the ``output'' characteristic operator by $
\widehat \Phi^{\mathrm{out}}_T(k):= U(T)^*\widehat\Phi_T(k)U(T)$. The key property giving the
commutativity of the observables in the Heisenberg picture is $ \widehat
\Phi^{\mathrm{out}}_T\left(1_{(0,t)}k\right)=\widehat \Phi^{\mathrm{out}}_t(k)$, \ $0\leq t
\leq T$. This property follows from the fact that we have $U(T)=U(T,t)U(t)$ (see Remark
\ref{Uts}), $\widehat \Phi_T\left(1_{(0,t)}k\right)=\widehat \Phi_t(k)$ and that $U(T,t)\in
\Uscr\left( \Hscr \otimes \Fscr_{(t,T)}\right)$ commutes with $\widehat \Phi_t(k)\in
\Uscr\left( \Fscr_{(0,t)}\right)$.

Let $\mathfrak{s}\in \Sscr(\Hscr\otimes \Fscr)$ be the initial system-field state. The
\emph{characteristic functional} of the process $\tilde X$ (the ``Fourier transform'' of its
probability law) is given by
\begin{equation}\label{eq:charac-fun}
\Phi_t(k) = \Tr \left\{\widehat\Phi_t(k)U(t)
\mathfrak{s}U(t)^*\right\}= \Tr
\left\{\widehat\Phi^{\mathrm{out}}_t(k) \mathfrak{s}\right\}.
\end{equation}
All the probabilities describing the continuous measurement of the observables $X(\alpha,t)$
are contained in $\Phi_t(k)$; let us give explicitly the construction of the joint
probabilities for a finite number of increments.

The measurable functions $\left\{\tilde X(\alpha,t;\cdot)\,,\; \alpha=1,\ldots,m, \; t\geq
0\right\}$, introduced in Theorem \ref{generalobservables}, represent the output signal of the
continuous measurement. Let us denote by $\Delta \tilde X(t_1,t_2)= \left(\tilde
X(1,t_2)-\tilde X(1,t_1), \ldots ,\tilde X(m,t_2)-\tilde X(m,t_1)\right)$ the vector of the
increments of the output in the time interval $(t_1,t_2)$ and by $\xi(\rmd \boldsymbol{x};
t_1,t_2)$ the joint projection valued measure on $\mathbb{R}^m$ of the increments
$X(\alpha,t_2)- X(\alpha,t_1)$, $\alpha=1,\ldots,m$. Note that, because of the properties of
the characteristic operator, not only the different components of an increment are commuting,
but also increments related to different time intervals; this implies that the projection
valued measures related to different time intervals commute. Moreover, the localisation
properties of the characteristic operator give
\begin{equation}
\xi(A; t_1,t_2)\equiv \xi\big(\Delta \tilde X(t_1,t_2)\in A\big) \in \Lscr(\Fscr_{(t_1,t_2)})\,, \quad
\text{for any Borel set}\ A\subset \mathbb{R}^m.
\end{equation}

As in the last part of Theorem \ref{generalobservables}, let us consider $0=t_0<t_1<\cdots <
t_n \leq t$, $k_\alpha(s)= \sum_{l=1}^n 1_{(t_{l-1},t_l)}(s) \,\kappa_\alpha^l$; then, we can
write
\begin{multline*}
\Phi_t(k) = \Tr\left\{\exp\biggl( \rmi \sum_{l=1}^n \sum_{\alpha=1}^m \kappa_\alpha^l \big[
X(\alpha,t_l) - X(\alpha,t_{l-1})\big]\biggr)U(t) \mathfrak{s}U(t)^*\right\}
\\
{}= \int_{\mathbb{R}^{nm}} \biggl(\prod_{l=1}^n  \rme^{\rmi\sum_{\alpha=1}^m \kappa_{\alpha}^l
x_{\alpha}^l}\biggr)\mathbb{P}_{\mathfrak{s}}\big[\Delta \tilde X(t_0,t_1)\in \rmd x^1
,\ldots, \Delta \tilde X(t_{n-1},t_n)\in \rmd x^n\big],
\end{multline*}
where the physical probabilities are given by
\begin{multline*}
\mathbb{P}_{\mathfrak{s}}\big[\Delta \tilde X(t_0,t_1)\in A_1 ,\ldots, \Delta \tilde
X(t_{n-1},t_n)\in A_n\big]
\\ {}= \Tr \biggl\{\biggl( \prod_{l=1}^n \xi(A_j; t_{l-1},t_l)\biggr) U(t)
\mathfrak{s} U(t)^* \biggr\}.
\end{multline*}
Obviously, \ $\Phi_t(k)$ \ is the characteristic function of the physical probabilities \
$\mathbb{P}_{\mathfrak{s}}\big[\Delta \tilde X(t_0,t_1)\in A_1 ,\ldots, \Delta \tilde
X(t_{n-1},t_n)\in A_n\big]$ and it uniquely determines them.

The aim is now to reformulate the continuous measurement in terms of system $S_\Hscr$ alone,
when the initial state is
\begin{equation}\label{initialstate}
\mathfrak{s}=\rho_0 \otimes |\psi(f)\rangle \langle \psi(f)|, \qquad \rho_0 \in \Sscr(\Hscr), \quad f\in
L^2(\mathbb{R}_+;\Zscr).
\end{equation}

Let $U(t)$ be a unitary, strongly continuous right cocycle and let us define $U(t,s)$ by Eq.\
\eqref{eq:U(t,s)}. Let $\widehat \Phi_t(k)$ be the characteristic operator introduced in
Definition \ref{def:Phi} under Hypothesis \ref{H:theBs} and set
\begin{equation}\label{eq:oc(t,s)}
\widehat \Phi(k;s,t):=\widehat \Phi_t\big(1_{(s,+\infty)}k\big), \qquad 0\leq s \leq t.
\end{equation}
By the definitions \eqref{def:hatphi}, \eqref{eq:oc(t,s)} and the points (1)-(3) of Theorem
\ref{generalobservables} one gets easily $\widehat\Phi(0;s,t)=\1$,
$\widehat\Phi(k;r,t)=\widehat\Phi(k;r,s)\,\widehat\Phi(k;s,t)$ and that $\widehat\Phi(k;s,t)$
is strongly continuous in $s$ and $t$.

\begin{definition}\label{defi:Gfk}
Let us take $f\in\LII$ and $0\le s \le t$. The \emph{reduced characteristic operator} is the
unique operator $\Gcal_f(k;s,t):\Lscr(\Hscr)\to\Lscr(\Hscr)$ that satisfies, $\forall
u,v\in\Hscr$, $\forall X\in\Lscr(\Hscr)$,
\begin{equation}\label{eq:defG}
\big\la v\big|\Gcal_f(k;s,t)[X]u\big\ra= \big\la U(t,s)v\otimes \psi(f)\big|\big(X\otimes \widehat
\Phi(k;s,t)\big)U(t,s)u\otimes \psi(f)\big\ra.
\end{equation}
Then, $\Tcal_f(s,t):=\Gcal_f(0;s,t)$ represents the \emph{reduced evolution operator} for the
observables of $S_\Hscr$.
\end{definition}

\begin{theorem}\label{teo:G-evolution}
\ In the hypotheses above, the family of linear maps \ $\Gcal_f(k;s,t)$, \ $t\ge s\ge0$,
$f\in\LII$, $k\in L^\infty(\Real_+; \Real^d)$, has the following properties:
\begin{enumerate}

\item $\Gcal_f(k;s,s)=\1; \qquad \norm{\Gcal_f(k;s,t)}\leq 1$;
\item $\Gcal_f(k;s,t)$ is \emph{completely positive definite}, i.e., for all integers $n$,
    test functions $k^i$, vectors $\phi_i$ and operators $X_i$, one has
\[
\sum_{i,j=1}^n \big\la \phi_i\big|\Gcal_f(k^i-k^j;s,t)[X_i^*X_j]\phi_j\big\ra\geq 0\,;
\]

\item $\Gcal_f(k;s,t)$ is a $\sigma$-weakly continuous operator on $\Lscr(\Hscr)$ and it
    has a pre-adjoint $\Gcal_f(k;s,t)_*$ acting on the trace class on $\Hscr$;
\item for each $X\in\Lscr(\Hscr)$ the maps $t\mapsto\Gcal_f(k;s,t)[X]$, $s\mapsto
    \Gcal_f(k;s,t)[X]$ and $\kappa\mapsto \Gcal_f(\kappa k;s,t)[X]$ are continuous with
    respect to the $\sigma$-weak topology of $\Lscr(\Hscr)$;

\item $\forall u,v \in \Hscr$, $\forall X\in\Lscr(\Hscr)$,
\begin{equation*}
\big\la
    v\big|\Gcal_f(k;s,t)[X]u\big\ra{}= \big\la U(t,s)v\otimes \psi(f_{(s,t)})\big|\big(X\otimes
    \widehat \Phi(k;s,t)\big)U(t,s)u\otimes \psi(f_{(s,t)})\big\ra;
\end{equation*}
\item if $f(x)=g(x)$ for all $x\in(s,t)$, we get $\Gcal_f(k;s,t)=\Gcal_g(k;s,t)$; then,
    $\Gcal_f(k;s,t)$ is well defined for all $f\in
    \mathrm{L}^2_{\mathrm{loc}}(\mathbb{R};\mathscr{Z})$;
\item $\Gcal_f(k;r,s)\circ \Gcal_f(k;s,t)=\Gcal_f(k;r,t)$, \quad $0\leq r \leq s \leq t$;
\item for all $s,t\geq 0$ we have $\Gcal_f(k;s,s+t)=\Gcal_{f_s}(k_s;0,t)\Big|_{h\to h_s,
    \, b\to b_s , \, c\to c_s}$, where we have introduced the shifted functions
    $f_s(x)=f(x+s)$, $k_s(x)=k(x+s)$, $h_s(x)=h(x+s)$, $b_s(x)=b(x+s)$, $c_s(x)=c(x+s)$.
\end{enumerate}

Moreover, the evolution operator $\Tcal_f(s,t)$, $t\ge s\ge0$, $f\in\LII$, introduced in
Definition \ref{defi:Gfk}, enjoys  the properties:
\begin{description}
\item[(i)] $\Tcal_f(s,t)[\1]=\1$; \qquad $\Tcal_f(s,s)=\1; \qquad \norm{\Tcal_f(s,t)}= 1;$

\item[(ii)] $\Tcal_f(s,t)$ is a $\sigma$-weakly continuous operator on $\Lscr(\Hscr)$ and
    it has a pre-adjoint $\Tcal_f(s,t)_*$ acting on the trace class on $\Hscr$;
\item[(iii)] $\Tcal_f(s,t)$ is completely positive;
\item[(iv)] for each $X\in\Lscr(\Hscr)$ the maps $t\mapsto\Tcal_f(s,t)[X]$ and $s\mapsto
    \Tcal_f(s,t)[X]$ are continuous with respect to the $\sigma$-weak topology of
    $\Lscr(\Hscr)$;

\item[(v)] $\Tcal_f(r,s)\circ \Tcal_f(s,t)=\Tcal_f(r,t)$, \quad $0\leq r \leq s \leq t$;

\item[(vi)] if $f(x)=g(x)$ for all $x\in(s,t)$, we have $\Tcal_f(s,t)=\Tcal_g(s,t)$; then,
    $\Tcal_f(s,t)$ is well defined for all $f\in
    L^2_{\mathrm{loc}}(\mathbb{R};\mathscr{Z})$;
\item[(vii)] for all $s,t\geq 0$ we have $\Tcal_f(s,s+t)=\Tcal_{f_s}(0,t)$, where
    $f_s(x)=f(x+s)$.
\end{description}
\end{theorem}

\begin{proof}
The first statement of point (1) is immediate from the fact that $\widehat \Phi(k;s,s)=\1$.
The second  statement follows from $\norm{\Gcal_f(k;s,t)[X]}\leq \norm{X\otimes \widehat
\Phi(k;s,t)}=\norm{X}$; the first step is from the definition \eqref{eq:defG}, the unitarity
of $U(t,s)$ and the normalisation of the coherent vector $\psi(f)$, while the second step is
due to the unitarity of the characteristic operator.

By using $\widehat{\Phi}(k^i-k^j;t,s)=\widehat{\Phi}(-k^i;t,s)^*\;\widehat{\Phi}(-k^j;t,s)$
and the definition of $\Gcal_f(k;s,t)$, one gets immediately
\begin{multline*}
\sum_{i,j=1}^n
    \big\la \phi_i\big|\Gcal_f(k^i-k^j;s,t)[X_i^*X_j]\phi_j\big\ra
    \\
{}= \norm{\sum_{j=1}^n X_j \otimes \widehat{\Phi}(-k^j;t,s)\,U(t,s)\,\phi_j\otimes
\psi(f)}^2\geq 0,
\end{multline*}
which is point (2).

Any $\tau \in \Tscr(\Hscr)$ can be written as $\tau=\sum_n|u_n\ra\la v_n|$ for some choice of
the vectors $u_n$, $v_n$ in $\Hscr$. Then, we have
\begin{multline*}
\Tr_{\Hscr}\left\{\Gcal_f(k;s,t)[X]\tau\right\} =\sum_{n} \la v_n|\Gcal_f(k;s,t)[X]u_n\ra
\\
{}=\sum_{n}\big\la U(t,s)\,v_n\otimes\psi(f)\big|\left(X\otimes\widehat
\Phi(k;s,t)\right)U(t,s)\, u_n\otimes\psi(f)\big\ra
\\
{}=\Tr_{\Hscr\otimes\Fscr}\left\{\left(X\otimes\widehat \Phi(k;s,t)\right)
U(t,s)\left(\tau\otimes |\psi(f)\ra \la \psi(f) |\right)U(t,s)^*\right\}
\\
{}=:\Tr_{\Hscr}\left\{X\Gcal_f(k;s,t)_*][\tau]\right\},
\end{multline*}
which defines the pre-adjoint. The existence of the pre-adjoint of $\Gcal_f(k;s,t)$ implies
its $\sigma$-weak continuity \cite[Corollary of Theorem 1.13.2, p.\ 29]{Sakai} and this
completes the proof of point (3)

By point (1) $\Gcal_f(\kappa k;s,t)$ is bounded uniformly in $s$, $t$ and $\kappa$. By
Proposition 1.15.2 in [\citen{Sakai}], the weak and the $\sigma$-weak topologies are
equivalent on the bounded spheres; so, it is enough to prove the weak continuity. Let us set
$\phi_1:= U(t,s)\, v \otimes \psi(f)$, $\phi_2:= U(t,s)\, u \otimes \psi(f)$, $\tilde X:=
X\otimes \widehat \Phi(k;s,t)$. Then, we have
\begin{multline*}
\abs{\la v| \Gcal_f(k;s,t+\epsilon)[X]u\ra-\la v| \Gcal_f(k;s,t)[X]u\ra} \\ {} = \abs{\la
U(t+\epsilon,t)\phi_1|\tilde X \widehat \Phi(k;t,t+\epsilon)U(t+\epsilon,t)\phi_2\ra - \la
\phi_1|\tilde X \phi_2\ra}
\\ {}\leq
\abs{\la \left(U(t+\epsilon,t)-\1\right)\phi_1|\tilde X \widehat
\Phi(k;t,t+\epsilon)U(t+\epsilon,t)\phi_2\ra} \\ {}+ \abs{\la\phi_1|\tilde X \widehat
\Phi(k;t,t+\epsilon)\left(U(t+\epsilon,t)-\1\right)\phi_2\ra} + \abs{\la \phi_1|\tilde X
\left(\widehat \Phi(k;t,t+\epsilon)-\1\right)\phi_2\ra}
\\ {}\leq
\norm{\tilde X \widehat
\Phi(k;t,t+\epsilon)U(t+\epsilon,t)}\norm{\phi_2}\norm{\left(U(t+\epsilon,t)-\1\right)\phi_1}
\\ {}+ \norm{\tilde X \widehat
\Phi(k;t,t+\epsilon)}\norm{\phi_1}\norm{\left(U(t+\epsilon,t)-\1\right)\phi_2}\\ {} +
\norm{\tilde X }\norm{\phi_1}\norm{\left(\widehat\Phi(k;t,t+\epsilon)-\1\right)\phi_2} \leq
\norm{X}\Big\{\norm{u}\norm{\left(U(t+\epsilon,t)-\1\right)\phi_1} \\ {}+
\norm{v}\left[\norm{\left(U(t+\epsilon,t)-\1\right)\phi_2} +
\norm{\left(\widehat\Phi(k;t,t+\epsilon)-\1\right)\phi_2}\right]\Big\},
\end{multline*}
which gives the continuity in $t$. The continuity in $s$ can be proved in a similar way. By
similar steps we get
\begin{multline*}
\abs{\la v| \Gcal_f(\kappa'k;s,t)[X]u\ra-\la v| \Gcal_f(\kappa k;s,t)[X]u\ra}
\\  {}\leq
\norm{X} \norm{v} \norm{\left(\widehat\Phi(\kappa'k;s,t)-\widehat\Phi(\kappa
k;s,t)\right)U(t,s)u\otimes \psi(f)},
\end{multline*}
which gives the continuity in $\kappa$, due to point (3) in Theorem \ref{generalobservables}.
This ends the proof of point (4). Points (5) and (6) are immediate by the localisation
properties.

By using the identification $\psi(f)=\psi(f_{(0,r)})\otimes
\psi(f_{(r,s)})\otimes\psi(f_{(s,t)})\otimes\psi(f_{(t})$ and the localisation properties
    $\widehat\Phi(k;a,b)\in \Uscr(\Fscr_{(a,b)})$, $U(b,a)\in \Uscr(\Hscr\otimes     \Fscr_{(a,b)})$, we have
\begin{multline*}
\la v| \Gcal_f(k;r,s)\circ \Gcal_f(k;s,t)[X]u\ra \\ {}= \left\la U(s,r)\left( v\otimes
\psi(f_{(r,s)}) \right)\big| \left(\Gcal_f(k;s,t)[X]\otimes \widehat \Phi(k;r,s) \right)
U(s,r)\left( u\otimes \psi(f_{(r,s)}) \right)\right\ra
\\ {}= \big\la U(t,s)\left[U(s,r)\left( v\otimes \psi(f_{(r,s)}) \right)\otimes \psi(f_{(s,t)})\right]\big|
\\ \left(\left( X \otimes \widehat
\Phi(k;s,t)\right)\otimes \widehat \Phi(k;r,s) \right) U(t,s)\left[U(s,r)\left( u\otimes
\psi(f_{(r,s)}) \right)\otimes \psi(f_{(s,t)})\right]\big\ra
\\ {}= \left\la U(t,r)\left( v\otimes \psi(f_{(r,t)}) \right)\big| \left(X\otimes \widehat \Phi(k;r,t)
\right) U(t,r)\left( u\otimes \psi(f_{(r,t)}) \right)\right\ra
\\ {}=
\la v| \Gcal_f(k;r,t)[X]u\ra,
\end{multline*}
which gives point (7). Finally, by Eqs.\ \eqref{eq:Theta}, \eqref{eq:U(t,s)},
\eqref{eq:U-comp},
    \eqref{Phimatrixelements} we have
\begin{multline*}
\big\la v\big|\Gcal_f(k;s,s+t)[X]u\big\ra
\\
{}=\big\la U(t)v\otimes\psi(f_s)\big| \big(X\otimes \Theta_s\widehat
\Phi(k;s,s+t)\Theta^*_s\big)U(t)u\otimes\psi(f_s)\big\ra
\\
{}=\big\la U(t,0)v\otimes\psi(f_s)\big|\left(X\otimes \widehat \Phi(k_s;0,t)\Big|_{h\to h_s,\,
b\to b_s,\, c\to c_s} \right)U(t,0)u\otimes\psi(f_s)\big\ra,
\end{multline*}
and point (8) follows.

By the particularising the previous statements to the case $k=0$, we get the properties of the
evolution operator.
\end{proof}

The definition of the reduced characteristic operator has been given in such a way that it is
sufficient to construct the characteristic functional \eqref{eq:charac-fun} when the initial
state is given by Eq.\ \eqref{initialstate}: $\Phi_t(k) = \Tr \left\{\Gcal_f(k;0,t)[\1]\rho_0
\right\}$. So, the reduced characteristic operator determines all the probabilities of the
output. However, the reduced characteristic operator gives something more: the states after
the measurement, conditional on the observed output. This is obtained through the
correspondence with the \emph{instruments} representing the continuous measurement, see
\cite[pp.\ 244--245]{BarSpringer} and [\citen{Bar86LNP}].

\section{The evolution equations}

Up to now, we have only made use of the cocycle properties of $U(t)$, but we are interested in
finding the infinitesimal generator and the evolution equation of the reduced characteristic
operator and for that we need also the QSDE for $U(t)$. The reduced characteristic operator
comes out from the product of three terms: the operators $\widehat \Phi_t(k)$, $U(t)$ and
$U(t)^*$. To compute the differential of this product we have to use two times the second
fundamental formula of quantum stochastic calculus.

Our first step will be to differentiate the unitary process
\begin{equation}\label{eq:Pok}
\Psi_t(k):=\left(\1\otimes\widehat \Phi_t(k)\right)U(t), \qquad k\in L^\infty(\Real_+; \Real^m);
\end{equation}
then, we shall use the second fundamental formula of quantum stochastic calculus to elaborate
the expression giving the reduced characteristic operator.

\begin{lemma}\label{prop:dif-Pok}
Let Hypotheses \ref{hyp:C}, \ref{hyp:Markov}, \ref{H:theBs} hold and the functions $c(t)$,
$b(t)$, $h^\alpha(t)$ be locally bounded in time. Then, $\Psi_t(k)$, defined by
\eqref{eq:Pok}, can be expressed as the quantum stochastic integral on
$\mathrm{Dom}(C^{1/2})\odot\Escr$
\begin{equation}\label{eq:dif-Pok}
\Psi_t(k)=\1+\sum_{i,j\ge0}\int_0^t\big(\1\otimes\widehat \Phi_s(k)\big)
M_{ij}(s;k)U(s)\rmd\Lambda_{ij}(s),
\end{equation}
where
\begin{subequations}\label{Mijexpl}
\begin{multline}
M_{00}(t;k) =K+\sum_{r=1}^d \langle \mathsf{r}(-k;t) | z_r\rangle R_r+\biggl\{\langle b(t)|
\left(\mathsf{S}\big(k(t)\big)-\1  \right)b(t)\rangle
\\ {}+\rmi \sum_{\alpha=1}^m k_\alpha(t) c^\alpha(t) - \frac 1 2 \sum_{\alpha,\beta=1}^m
k_\alpha(t) \langle h^\alpha(t) | h^\beta(t) \rangle k_\beta(t)\biggr\}\1,
\end{multline}
\begin{equation}
M_{0j}(t;k)=N_j +\sum_{r=1}^d \langle \mathsf{r}(-k;t) | z_r\rangle S_{rj}\,, \qquad j\geq
1\,,
\end{equation}
\begin{equation}
M_{i0}(t;k)=\sum_{r\geq 1}\left\langle z_i \big| \mathsf{S}\big(k(t)\big)z_r \right\rangle
R_{r}+\langle z_i |\mathsf{r}(k;t)\rangle \1, \qquad i\geq 1\,,
\end{equation}
\begin{equation}
M_{ij}(t;k)=\sum_{r\geq 1} \left\langle z_i\big|\mathsf{S}\big(k(t)\big)z_r\right\rangle
S_{rj}-\delta_{ij}\1, \qquad i,j\geq 1,
\end{equation}
\end{subequations}
with $\mathrm{Dom}\big(M_{00}(t;k)\big)=\mathrm{Dom}(K)$, $\mathrm{Dom}\big(M_{0j}(t;k)\big)=
\mathrm{Dom}(N_j)\supset \mathrm{Dom}(K)$, $\mathrm{Dom}\big(M_{i0}(t;k)\big)\supset
\bigcap_{k=1}^d\mathrm{Dom}(R_k)\supset \mathrm{Dom}(K)\cup \mathrm{Dom}(K^*)$,
$\mathrm{Dom}\big(M_{ij}(t;k)\big)=\Hscr$, $i,j=1,\ldots, d$.

Moreover, $\forall f,g\in\Mscr$ and $\forall u,v\in \mathrm{Dom}(C^{1/2})$, one has
\begin{multline}\label{eq:dif-UPok}
\langle U(t)v\otimes e(g)|\left(X\otimes \widehat \Phi_t(k)\right)U(t) u\otimes e(f) \rangle
=\langle v|Xu \rangle \langle e(g)|e(f) \rangle
\\ {}+
\sum_{i,j\geq 0}\int_0^t\rmd s \, \overline{g_i(s)}\Bigl\{\big\langle U(s) v \otimes e(g)
\big|\big(X\otimes\widehat \Phi_s(k)\big) M_{ij}(s;k) U(s) u \otimes e(f)\big\rangle
\\ {}+
\langle F_{ji}U(s)v\otimes e(g)|\big(X\otimes\widehat \Phi_s(k)\big)U(s)u\otimes e(f) \rangle
\\ {}+
\sum_{l\geq 1}\big\langle F_{li}U(s)v\otimes e(g) \big|\big(X\otimes\widehat \Phi_s(k)\big)
M_{lj}(s;k) U(s) u \otimes e(f)\big\rangle \Bigr\}f_j(s).
\end{multline}
\end{lemma}
Let us recall the convention $f_0(s)=g_0(s)=1$.
\begin{proof} By Eqs.\ \eqref{eq:rHP} and \eqref{PhiQSDE}, the second fundamental formula of
quantum stochastic calculus, $\Phi_t(-k)=\Phi_t(k)^*$ and $\overline{G_{ji}(s;-k)}=
G_{ij}(t;k)$, we get for $f,g\in\Mscr$ and $u,v\in \mathrm{Dom}(C^{1/2})$
\begin{multline}\label{eq:dPU}
\big\la v\otimes e(g)\big|\Psi_t(k)u\otimes e(f)\big\ra - \langle v| u\rangle \langle
e(g)|e(f) \rangle
\\ {}=
\sum_{i,j\geq 0}\int_0^t\rmd s \, \overline{g_i(s)}\big\langle v \otimes \widehat
\Phi_t(-k)e(g) \big| M_{ij}(s;k) U(s) u \otimes e(f)\rangle f_j(s),
\end{multline}
where
\begin{equation}\label{Mij}
M_{ij}(s;k):=F_{ij}+G_{ij}(s;k)\1 +\sum_{r\geq 1}G_{ir}(s;k)F_{rj}\,.
\end{equation}
By inserting the explicit expressions of the elements of the matrices $F$ and $G$ into Eq.\
\eqref{Mij} we get Eqs.\ \eqref{Mijexpl}. The statements about the domains follow from
Hypothesis \ref{hyp:C} point (iv), Proposition \ref{prop:moreonF} point (1) and the fact that
the operators $S_{ij}$ are bounded.

It is easy to check that the processes $\big(\1\otimes\widehat \Phi_s(k)\big) M_{ij}(s;k)U(s)$
are stochastically integrable, by using the fact that $\widehat \Phi_s(k)$ is unitary, the
functions $G_{ij}(s;k)$ are locally bounded, due to the boundedness assumption on
$c,\,b,\,h^\alpha$, and the processes $ F_{ij}U(s)$ are stochastically integrable by
hypothesis. Then, Eq.\ \eqref{eq:dif-Pok} follows from Eq.\ \eqref{eq:dPU} and the first
fundamental formula of quantum stochastic calculus.

By the second fundamental formula of quantum stochastic calculus applied to $(X^*\otimes
\1)U(t)$ and $\Psi_t(k)$ we get immediately Eq.\ \eqref{eq:dif-UPok}.
\end{proof}

For $\lambda,\, \mathsf{r} \in \Zscr$ (with components denoted by $\lambda_j$ and
$\mathsf{r}_j$) let us define the operators
\begin{subequations}
\begin{gather}
B_i(\lambda):= R_i + \sum_{j=1}^d S_{ij}\lambda_j, \qquad i=1,\ldots,d,
\\
K(\lambda,\mathsf{r}):= K- \sum_{i,j=1}^d R_i^{\,*}S_{ij}\lambda_j -
\frac{\norm{\lambda}^2}2\,\1 + \sum_{i=1}^d \overline{\mathsf{r}_i}\,B_i(\lambda).
\end{gather}
\end{subequations}
By taking into account Hypothesis \ref{hyp:C} and Proposition \ref{prop:moreonF} we have
\begin{gather*}
\mathrm{Dom}\big(B_i(\lambda)\big)=\mathrm{Dom}(R_i) \supset \mathrm{Dom}(K)\cup
\mathrm{Dom}(K^*), \\ \mathrm{Dom}\big(K(\lambda,\mathsf{r})\big)=\mathrm{Dom}(K)\supset
\mathrm{Dom}(C^{1/2}).
\end{gather*}

Again by Hypothesis \ref{hyp:C} and Proposition \ref{prop:moreonF}, the domains of the adjoint
of the previous operators contain $\mathrm{Dom}(F^*)\supset \widetilde D$ and on
$\mathrm{Dom}(F^*)$ we have
\begin{subequations}\label{B*K*}
\begin{gather}
B_i(\lambda)^*= R_i^{\;*} + \sum_{j=1}^d \overline{\lambda_j}\, S_{ij}^{\;*}, \qquad
i=1,\ldots,d,
\\
K(\lambda,\mathsf{r})^*= K^*- \sum_{i,j=1}^d \overline{\lambda_j}\, S_{ij}^{\;*} R_i -
\frac{\norm{\lambda}^2}2\,\1 + \sum_{i=1}^d \mathsf{r}_i\,B_i(\lambda)^*.
\end{gather}
\end{subequations}

Finally, for $\kappa,\, c \in \mathbb{R}^m$, $b\in \Zscr$, $h\in \Zscr^m$ we define also
\begin{equation}\label{C()}
C(\kappa, b, c, h):= \langle b|\bigl(\mathsf{S}(\kappa)-\1\bigr)b\ra+\rmi \sum_{\alpha= 1}^m
\kappa_\alpha c^\alpha- \frac 1 2  \sum_{\alpha,\beta= 1}^m \kappa_\alpha\la
h^\alpha|h^\beta\ra \kappa_\beta\,.
\end{equation}

\begin{proposition}\label{prop:step1}
Let Hypotheses \ref{hyp:C}, \ref{hyp:Markov}, \ref{H:theBs} hold and the functions $c(t)$,
$b(t)$, $h^\alpha(t)$ be locally bounded in time. Then, $\forall f\in \Mscr$, $\forall k\in
L^\infty(\Real_+; \Real^m)$, $\forall u,v\in \mathrm{Dom}(C^{1/2})$ we have
\begin{multline}
\langle U(t)v\otimes \psi(f)|\left(X\otimes \widehat \Phi_t(k)\right)U(t) u\otimes \psi(f)
\rangle =\langle v|Xu \rangle
\\ {}+
\int_0^t\rmd s \, \biggl\{\big\langle U(s) v \otimes \psi(f) \big|\big(X\otimes\widehat
\Phi_s(k)\big) K\big(f(s),\mathsf{r}(-k,s)\big) U(s) u \otimes \psi(f)\big\rangle
\\ {}+
\langle K\big(f(s),\mathsf{r}(k,s)\big)U(s)v\otimes \psi(f)|\big(X\otimes\widehat
\Phi_s(k)\big)U(s)u\otimes \psi(f) \rangle
\\ {}+
\sum_{i,j= 1}^d \langle z_i|\mathsf{S}\big(k(s)\big)z_j\ra \big\langle
B_i\big(f(s)\big)U(s)v\otimes\psi(f) \big|\big(X\otimes\widehat \Phi_s(k)\big)
B_j\big(f(s)\big) U(s) u \otimes \psi(f)\big\rangle
\\ {}+
C\big(k(s),b(s),c(s),h(s)\big) \big\langle U(s)v\otimes\psi(f) \big|\big(X\otimes\widehat
\Phi_s(k)\big) U(s) u \otimes \psi(f)\big\rangle \biggr\}.
\end{multline}
\end{proposition}
\begin{proof} The statement follows by direct computations,
by inserting the explicit expressions of $F_{ij}$, $M_{ij}(t;k)$, $N_j$ into Eq.\
\eqref{eq:dif-UPok} with $g=f$.
\end{proof}

Let $\Dscr \subset \Lscr(\Hscr)$ be the linear span of the rank-one operators of the type
$|\psi\rangle \langle \phi|$ with $\psi,\,\phi \in \mathrm{Dom}(F^*)$. By using the operators
\eqref{B*K*}, we define, $\forall \psi,\,\phi \in \mathrm{Dom}(F^*)$, $\forall u,\, v \in
\Hscr$,
\begin{multline}\label{defKcal}
\langle v|\Kcal_f^k(t)[|\psi\rangle \langle \phi|] u\rangle =\langle  v |\psi\rangle \langle
K\big(f(t),\mathsf{r}(-k,t)\big)^* \phi|u \rangle+\langle  v
|K\big(f(t),\mathsf{r}(k,t)\big)^*\psi\rangle \langle \phi|u \rangle
\\ {}+
\sum_{i,j= 1}^d \langle z_i|\mathsf{S}\big(k(t)\big)z_j\ra \langle v|
B_i\big(f(t)\big)^*\psi\rangle\langle  B_j\big(f(t)\big)^*\phi| u \rangle
\\ {}+
C\big(k(t),b(t),c(t),h(t)\big) \langle v| \psi\rangle \langle \phi|u\rangle;
\end{multline}
then, by linearity, we extend $\Kcal_f^k(t)$ to $\Dscr$.

\begin{corollary}\label{cor:step2}
Let Hypotheses \ref{hyp:C}, \ref{hyp:Markov}, \ref{H:theBs} hold and the functions $c(t)$,
$b(t)$, $h^\alpha(t)$ be locally bounded in time. Then, $\forall f\in \Mscr$, $\forall k\in
L^\infty(\Real_+; \Real^m)$, $\forall u,v\in \Hscr$, $\forall X\in \Dscr$, we have
\begin{equation}\label{eq:Gfkeqint}
\la v\big|\Gcal_f(k;0,t)[X]u\ra=\la v|Xu\ra +
\int_0^t\big\la v\big|\Gcal_f(k;0,s)\big[\Kcal^{k}_{f}(s)[X]\big]u\big\ra\rmd s\,.
\end{equation}
\end{corollary}
\begin{proof}
By using the notations above, Proposition \ref{prop:step1} gives immediately Eq.\
\eqref{eq:Gfkeqint} $\forall u,v\in \mathrm{Dom}(C^{1/2})$. Being $X\in\Dscr$, the operator
$\Kcal^{k}_{f}(s)[X]\big]$ turns out to be bounded; moreover, we have
$\norm{\Gcal_f(k;s,t)}\leq 1$. Then, we can extend \eqref{eq:Gfkeqint} to any $u,\,v\in
\Hscr$.
\end{proof}

By introducing the pre-adjoint of $\Gcal_f(k;0,t)$ and extending \eqref{eq:Gfkeqint} to the
whole trace class we get: $\forall X\in \Dscr$, $\forall \tau \in \Tscr(\Hscr)$,
\begin{equation}\label{foreward}
\Tr_\Hscr\left\{X\Gcal_f(k;0,t)_*[\tau]\right\}=\Tr_\Hscr\left\{X\tau\right\}+
\int_0^t\Tr_\Hscr\left\{\Kcal^{k}_{f}(s)[X]\Gcal_f(k;0,s)_*[\tau]\right\}\rmd s\,,
\end{equation}
with initial condition $\Gcal_f(k;0,0)_*=\1$. For $k=0$ and  $\tau \in \Sscr(\Hscr)$, Eq.\
\eqref{foreward} is a \emph{quantum master equation} and the formal pre-adjoint of
$\Kcal^{0}_{f}(t)$ is known as \emph{Liouville operator}. We can say that \eqref{foreward} is
a generalisation of a quantum master equation, which includes the continuous measurement.

The problem which remains open is to prove the uniqueness of the solution of Eq.\
\eqref{eq:Gfkeqint} or of Eq.\ \eqref{foreward}. We note that in the case of quantum dynamical
semigroups the positivity plays a role in the analogous problem, see Theorem
\ref{teo:QDS-min}, while in the case of Eq.\ \eqref{eq:Gfkeqint} we have only that $\Gcal$ is
positive definite in $k$.

\section{An example: the degenerate parametric oscillator}
The degenerate parametric oscillator is the physical system which was used to produce squeezed
light.\cite{WuK86,GraS87,WuK87} The squeezing of the light was revealed by balanced homodyne
detection, a measurement scheme which is indeed described by continuous measurements of
diffusive type.\cite{Bar90QO,BarGSpringer}. Such a quantum optical system is constituted by an
optical cavity closed by two partially transparent mirrors with inside a crystal with a
$\chi^{(2)}$ non-linearity. Only two cavity modes of the electromagnetic field inside the
cavity are relevant: the subharmonic field of frequency $\omega_C$ (a quantum oscillator with
annihilation and creation operators $a,\, a^\dagger$) and the pump field of double frequency
(with annihilation and creation operators $b,\, b^\dagger$). The pump field is populated by a
resonant laser entering the cavity, the crystal couples the two modes and the light coming out
of the cavity is detected by homodyne devices and/or photocounters. The degenerate parametric
oscillator is well studied from the point of view of theoretical physics and quantum optics in
\cite[Chapts.\ 9, 10, 12, 18]{Carm08}. Here we want to prove that the mathematical model of
the degenerate parametric oscillator with direct and homodyne detection can be rigourously
formulated and gives an example of the theory we have developed.

The formal master equation is given in \cite[Eq.\ (9.97)]{Carm08} and reads
\begin{multline}\label{mastereqC}
\dot \rho(t) =-\rmi [H_0,\rho(t)]-\rmi \left[\lambda\rme^{-2\rmi \omega_C t}b^\dagger+
\overline{\lambda}\rme^{2\rmi \omega_C t} b ,\,\rho(t)\right]
\\
{}+\kappa\left(\overline{n}+1\right)\left(2a\rho(t) a^\dagger-a^\dagger a \rho(t)-
\rho(t)a^\dagger a \right)
\\{}+
\kappa\overline{n}\left(2a^\dagger\rho(t) a-a a^\dagger \rho(t)- \rho(t)a a^\dagger \right)+
\kappa_p\overline{n}_p\left(2b^\dagger\rho(t) b-b b^\dagger \rho(t)- \rho(t)b b^\dagger
\right)
\\
{}+\kappa_p\left(\overline{n}_p+1\right)\left(2b\rho(t) b^\dagger-b^\dagger b \rho(t)-
\rho(t)b^\dagger b \right).
\end{multline}
The Hamiltonian term $H_0$ contains the free energies of the modes and the interaction due to
the $\chi^{(2)}$ non-linearity:
\begin{equation}
H_0=\omega_C a^\dagger a +2\omega_C b^\dagger b+ \frac{\rmi g}2\left(a^{\dagger 2}b- b^\dagger
a^2\right).
\end{equation}
For the constants we have $ \omega_C>0$, \ $g\in\Real$, \ $g\neq 0$, \ $\kappa>0$, \
$\overline{n}\geq 0$, \ $\kappa_p>0$, \ $\overline{n}_p\geq 0$.

This model, plus detection, can be rigourously formulated in the set up developed before.
First, the Hilbert space is identified with the span of the eigenvectors of the two number
operators and the creation and annihilation operators are defined. Let us take the Hilbert
space $\Hscr=\ell^2(\mathbb{N})\otimes \ell^2(\mathbb{N})$ with canonical orthonormal basis
$\{e_{n,m}, \;n,m\geq 0\}$.  The creation, annihilation and number operators for the
subharmonic mode are defined by
\begin{subequations}
\begin{gather}\label{Dom_a}
\mathrm{Dom}(a) = \mathrm{Dom}\left(a^\dagger\right) = \left\{u\in\Hscr :
\sum_{n,m\geq 0}n \abs{u_{n,m}}^2<+\infty\right\},
\\
a^\dagger\,e_{n,m}=\sqrt{n+1}\,e_{n+1,m}, \qquad a\,e_{0,m}=0, \quad
a\,e_{n,m}=\sqrt{n}\,e_{n-1,m}, \ \text{if } n>0,
\\
\mathrm{Dom}\left(a^\dagger a\right) = \left\{u\in\Hscr : \sum_{n,m\geq 0}n^2
\abs{u_{n,m}}^2<+\infty\right\}, \qquad a^\dagger a\,e_{n,m}=n\,e_{n,m}.
\end{gather}
\end{subequations}
An analogous definition holds for the operators $b^\dagger,\, b,\, b^\dagger b$, which act on
the second factor of the tensor product.

In constructing the model we have to reproduce the effects contained in the master equation
\eqref{mastereqC} and to introduce the measurement. So, we have to introduce losses at the
mirrors and thermal dissipation in the crystal and at the walls of the cavity. We have also to
introduce the possibility of injecting laser light feeding the pump mode. Moreover, we
consider the direct detection of photons with two photocounters, chosen one to be sensible to
photons around frequency $\omega_C$ and the other to photons around frequency $2\omega_C$.
Finally, we consider homodyning around frequency $\omega_C$. To realise all these features in
the mathematical model we need many channels, but some channels with similar structure can be
collected together and the minimal number is $d=8$. We use channels 1 and 2 to describe the
light reaching the two photocounters and channel 3 for the light reaching the homodyne
detector, channel 4 is the one used for the injection of the laser, channels 5--8 describe
losses and thermal dissipation. There is no scattering effect which mixes the channels. The
channel operators and the unitary matrix of system operators we need are
\begin{subequations}
\begin{gather}
R_1 =\beta_{1}\,b, \qquad R_2 =\alpha_{1}\,a, \qquad R_3 =\alpha_{2}\,a, \qquad R_4
=\beta_{2}\,b, \qquad R_5 =\beta_3\,b,
\\
R_6 =\alpha_{3}\,a, \qquad R_7 = \beta_4\,b^\dagger,\qquad R_8 =\alpha_4\,a^\dagger, \qquad
S_{ij}=\delta_{ij}\1,
\\
\abs{\alpha_{1}}^2 +\abs{\alpha_{2}}^2 +\abs{\alpha_{3}}^2=
2\kappa\left(\overline{n}+1\right), \qquad \abs{\alpha_{4}}^2= 2\kappa\overline{n},
\\
\abs{\beta_{1}}^2 +\abs{\beta_{2}}^2 +\abs{\beta_{3}}^2=
2\kappa_p\left(\overline{n}_p+1\right), \qquad \abs{\beta_{4}}^2= 2\kappa_p\overline{n}_p.
\end{gather}
\end{subequations}

The operator $K$ has to include the Hamiltonian $H_0$ and to satisfy Eq.\
\eqref{dissipativity}; so, it must have the formal expression
\begin{multline*}
K=-\rmi H_0-\frac 1 2 \sum_{i=1}^8 {R_i}^*R_i =\frac{g}2\left(a^{\dagger 2}b- b^\dagger
a^2\right)-\left(\kappa \overline{n} +\kappa_p\overline{n}_p\right)\1
\\{}-\left(\rmi \omega_C+\kappa \left(2 \overline{n} + 1 \right)\right)
a^\dagger a - \left(2\rmi \omega_C+\kappa_p \left( 2\overline{n}_p + 1 \right)\right)
b^\dagger b .
\end{multline*}
Rigourously, by defining $u_{n,m}=0$ if $n<0$ and/or $m<0$, we have
\begin{subequations}\label{Def_K}
\begin{equation}
Ku = \sum_{n,m}k_u(n,m)\,e_{n,m}\,, \qquad \mr{Dom}(K)=\left\{u:\sum_{n,m}\abs{k_u(n,m)}^2<+\infty\right\},
\end{equation}
\begin{multline}
k_u(n,m) := \frac g2\,\sqrt{n(n-1)(m+1)}\, u_{n-2,m+1} - \frac g2\,\sqrt{m(n+1)(n+2)}\,
u_{n+2,m-1}
\\ {}-
\left[\kappa \overline{n}+ \kappa_p\overline{n}_p +\rmi \omega_Cn +\kappa \left(2
\overline{n}+ 1  \right) n +2\rmi \omega_Cm +\kappa_p \left(2 \overline{n}_p+1  \right)
m\right]u_{n,m}\,.
\end{multline}
\end{subequations}

\begin{theorem}\label{theo:hypF}
Let us construct the $F$-matrix by setting $F_{00}=K$, $F_{i0}=R_i$, $F_{0j}=
N_j=:-R_j^{\,*}$, $F_{ij}=0$, $i,j\geq 1$. Then, Hypothesis \ref{hyp:C} hold true with $D=\til
D$ given by the linear span of the basis $\{e_{n,m},\; n,m\geq 0\}$ and with $C=N^4$, where $
N:=a^\dagger a +2b^\dagger b$.
\end{theorem}
\begin{proof}
By applying the definition of adjoint and Riesz lemma \cite{Pazy83} one can easily check that
$a^*=a^\dagger$, $a^{\dagger\,*}=a$ and the same for $b$, $b^\dagger$, as it is well known. In
particular all operators $R_i$, $N_i$ are closed \cite{Fagnola02}. By \eqref{Dom_a} we have
the domain
\begin{equation}\label{Dom_ab}
D_{RN}:=\bigcap_{\begin{smallmatrix}{i,\;j}\\ { (i,j)\neq (0,0)}\end{smallmatrix}}
\mathrm{Dom}\left(F_{ij}\right)
= \left\{u\in\Hscr :\sum_{n,m\geq 0}\left(n+m\right) \abs{u_{n,m}}^2<+\infty\right\}.
\end{equation}

Again by the definition of adjoint, \cite{Pazy83} we get $K^*$, which turns out to be defined
by Eqs.\ \eqref{Def_K} with the substitutions $\omega_C\to -\omega_C$, $g\to -g$.
%\[
%K^*u = \sum_{n,m}k_u^*(n,m)\,e_{n,m}\,, \qquad \mr{Dom}(K^*)=
%\left\{u:\sum_{n,m}\abs{k_u^*(n,m)}^2<+\infty\right\},
%\]
%\begin{multline*}
%k_u^*(n,m) := -\frac g2\,\sqrt{n(n-1)(m+1)}\, u_{n-2,m+1} + \frac g2\,\sqrt{m(n+1)(n+2)}\,
%u_{n+2,m-1}
%\\ {}-
%\left[\frac 1 2 \left(\kappa \overline{n}+ \kappa_p\overline{n}_p\right) -\rmi \omega_Cn
%+\kappa \left( \overline{n}+\frac 1 2 \right) n -2\rmi \omega_Cm +\kappa_p \left(
%\overline{n}_p+\frac 1 2 \right) m\right]u_{n,m}\,.
%\end{multline*}
From the definitions of $K$ and $K^*$ we get, by direct computations, the dissipativity
conditions \eqref{dissipativity} and $K^{**}=K$. In particular also $K$ is closed.

For every $u\in D$ we get, from the dissipativity condition
\[
\norm{Ku}^2\geq \Big\langle Ku\Big| \frac {u}{\norm{u}}\Big\rangle \Big\langle \frac
{u}{\norm{u}} \Big| Ku\Big\rangle = \frac{\abs{\langle Ku| u\rangle}^2}{\norm{u}^2} \geq
\frac{\left(\RE\langle Ku| u\rangle\right)^2}{\norm{u}^2} =
\frac{\left(\sum_k\norm{R_ku}^2\right)^2}{4\norm{u}^2},
\]
which gives $\norm{R_k u}^4 \leq 4 \norm{u}^2 \norm{Ku}^2$ and $D_{RN}\supset
\mathrm{Dom}(K)$. Analogously, we get $D_{RN}\supset \mathrm{Dom}(K^*)$. Up to now we have
proved conditions (i), (ii), (iv), (v), (vii).

To prove condition (iii) we have to show that the set $D$ given in the Proposition is a core
for $K$, $K^*$, $a$, $a^\dagger$, $b$, $b^\dagger$. By $\langle u |a e_{n,m}\rangle =
\sqrt{n}\; \overline {u_{n-1,m}}$, we get
\[
\left\{ u\in \Hscr: {}\ \exists \phi \in \Hscr : \langle u| a e_{n,m}\rangle = \langle \phi |
e_{n,m}\rangle \ \forall n,m\right\}= \mathrm{Dom}(a^\dagger).
\]
Similar considerations hold for $a^\dagger$, $b$, $b^\dagger$ and $D$ is a core for $a$,
$a^\dagger$, $b$, $b^\dagger$. Analogously, by writing the expressions of $Ke_{n,m}$ e
$K^*e_{n,m}$ we deduce that $D$ is a core for $K$ and $K^*$.

By \eqref{dissipativity} both $K$ e $K^*$ are \emph{dissipative} \cite[Definition 4.1 p.\
13]{Pazy83} and, by \cite[Corollary 4.4 p.\ 15]{Pazy83}, they are generators of contraction
semigroups in $\Hscr$. This complete the proof of point (vi).

Finally, let us consider point (viii).

We have $\Dom\big(C^{1/2}\big)=\left\{u\in\Hscr : \sum_{n,m\geq 0}(n+2m)^4
\abs{u_{n,m}}^2<+\infty\right\}$, which is obviously contained in $\Dom(F)=\Dom(K)$ given in
\eqref{Def_K}.

For any $\epsilon>0$ take $D_\epsilon=D=\widetilde D$. Then, $C_\epsilon^{1/2}
{D}_\epsilon\subset\til{D}$ because $C_\epsilon$ is diagonal in the canonical basis. For large
$n$ and/or $m$, $C_\epsilon\, e_{n,m}$ goes as $1/\big(\epsilon(n+2m)^2\big)$ and each
operator $F_{ij}^*C_\epsilon^{1/2}|_{{D}_{\epsilon}}$ is bounded (the worst case is for
$F_{00}^*=K^*$). This is point (a).

By explicitly computing the left hand side of the inequality in point (b), we see that we have
to prove the inequality
\[
2\RE \sum_{i\geq 1}\Bigl\langle u_i+\frac 1 2 \, R_iu_0 \Big|[C_\epsilon,R_i]u_0\Bigr\rangle
\leq  \sum_{i\geq 0}\left(b_1
\la u_i| C_\epsilon u_i\ra+ b_2 \norm{u_i}^2\right).
\]
The proof of this inequality is long and we give only a sketch.

Let us note that for any function of the operator $N$ one has $ af(N)=f(N+1)a$, \
$a^{\dag}f(N+1)=f(N)a^{\dag}$, \ $bf(N)=f(N+2)b$, \ $b^{\dag}f(N+2)=f(N)b^{\dag}$. \ We take
$f(x)= \frac {x^4}{\left(1+\epsilon x^4\right)^2}$ for $x\geq 0$, $f(x)=0$ for $x<0$, so that
$f(N)=C_\epsilon$.

By these relations and standard estimates one can obtain
\[
\RE \langle \alpha a u |[C_\epsilon, \alpha a]u\rangle \leq \abs{\alpha}^2 \langle a^\dagger a
u| \abs{f(N-1)-f(N)} u\rangle.
\]
By the relations above, expansion in the canonical basis and standard estimates we can prove
also
\[
2\RE \langle v |[C_\epsilon, \alpha a]u\rangle \leq \abs{\alpha}^2 \langle a^\dagger a
u| \abs{f(N-1)-f(N)} u\rangle + \langle v | \abs{f(N)-f(N+1)} v\rangle .
\]
Analogous estimates can be obtained in the cases involving $a^\dagger$, $b$, $b^\dagger$. All
together these results give
\begin{multline*}
2\RE \sum_{i\geq 1}\Bigl\langle u_i+\frac 1 2 \, R_iu_0 \Big|[C_\epsilon,R_i]u_0\Bigr\rangle
\leq   \sum_{i=1,4,5} \big\langle u_i\big|\abs{f(N+2)-f(N)}u_i\big\rangle
\\ {}+4\kappa\left(\overline n +1\right)\big\langle a^\dagger a
u_0\big|\abs{f(N-1)-f(N)}u_0\big\rangle+\big\langle a^\dagger a
u_8\big|\abs{f(N-1)-f(N)}u_8\big\rangle
\\ {}+
4\kappa\overline n \big\langle \left(a^\dagger a+1\right)
u_0\big|\abs{f(N+1)-f(N)}u_0\big\rangle + \sum_{i=2,3,6} \big\langle
u_i\big|\abs{f(N+1)-f(N)}u_i\big\rangle
\\
{}+ 4\kappa_p\left(\overline{n}_p +1\right)\big\langle b^\dagger b
u_0\big|\abs{f(N-2)-f(N)}u_0\big\rangle+\big\langle b^\dagger b
u_7\big|\abs{f(N-2)-f(N)}u_7\big\rangle
\\ {}+
4\kappa_p\overline{n}_p \big\langle \left(b^\dagger b+1\right)
u_0\big|\abs{f(N+2)-f(N)}u_0\big\rangle.
\end{multline*}

By using the specific form of $f$ and  $2(x-1)\geq x$ for $x\geq 2$ and $3(x-2)\geq x$ for
$x\geq 3$, we get
\begin{gather*}
\abs{1-\frac{f(x+1)}{f(x)}}\leq \frac{15}{x}\,, \qquad \abs{1-\frac{f(x+2)}{f(x)}}\leq
\frac{80}{x}\,, \qquad \text{for } x\geq 1;
\\
\abs{\frac{f(x-1)}{f(x)}-1}\leq \frac{64}{x}\,, \quad \text{for } x\geq 2; \qquad
\abs{\frac{f(x-2)}{f(x)}-1}\leq \frac{648}{x}\,, \quad \text{for } x\geq 3.
\end{gather*}
These inequalities can slightly modified to include also the cases $x=0,1,2$. Then, by further
straightforward estimates, one gets
\begin{multline*}
2\RE \sum_{i\geq 1}\Bigl\langle u_i+\frac 1 2 \, R_iu_0 \Big|[C_\epsilon,R_i]u_0\Bigr\rangle
\leq   4\left(\kappa \overline{n} +16 \kappa_p \overline{n}_p\right) \norm{u_0}^2 + 324\langle
u_7|C_\epsilon u_7\rangle
\\ {} +
64 \langle u_8|C_\epsilon u_8\rangle+\sum_{i=2,3,6}\left(\norm{u_i}^2+15\langle u_i|C_\epsilon
u_i\rangle\right) +16 \sum_{i=1,4,5}\left(\norm{u_i}^2 +  5 \langle u_i|C_\epsilon
u_i\rangle\right)
\\
{}+8\left(32\kappa (\overline{n}+1)+ 162\kappa_p (\overline{n}_p+1) + 15 \kappa \overline{n}
+60 \kappa_p \overline{n}_p\right) \langle u_0|C_\epsilon u_0\rangle .
\end{multline*}
This ends the proof of the inequality.
\end{proof}

To use the number operator $ N=a^\dagger a +2b^\dagger b$, which commutes with $H_0$, is
suggested by \cite[Chapter 3]{Bott}, where the conservativity property of the minimal quantum
dynamical semigroup in a similar model is proved.

\begin{proposition} Hypothesis \ref{hyp:Markov} holds for the model of this section.
\end{proposition}
\begin{proof}
We prove the sufficient condition of Theorem \ref{theor_QZ} with $A=K^*$, $L_k=N_k^*=-R_k$.
Let $D$ be as in Theorem \ref{theo:hypF} and, on their maximal domains, let us introduce the
operators
\[
Q:= wN+2\kappa \overline{n}+2\kappa_p \overline{n}_p\,, \qquad w:= \max\left\{ 2\kappa\left( 2\overline{n}
+1\right),\; \kappa_p\left( 2\overline{n}_p+ 1\right)\right\},
\]
\[
Z:=2\kappa\left( 2\overline{n}+1\right)a^\dagger a+ 2\kappa_p\left( 2\overline{n}_p
+ 1\right) b^\dagger b+2\kappa \overline{n}+2\kappa_p \overline{n}_p\,.
\]
By defining also $v:= \min\left\{ 2\kappa\left( 2\overline{n} +1\right),\; \kappa_p\left(
2\overline{n}_p+ 1\right)\right\}$, on $D$ we have
\[
0\leq \frac vw\, Q + \frac{w-v}w\left(2\kappa \overline{n}+2\kappa_p \overline{n}_p\right)\leq Z\leq Q.
\]
In particular we get $\Dom(Q)=\Dom(Z)=\Dom(N)=\Dom(a^\dagger a)\cap \Dom(b^\dagger b)$,
$D\subset \Dom(Q)\subset \Dom\left(Q^{1/2}\right)=\Dom\left(Z^{1/2}\right)$. The set $D$ is a
core for $Q^{1/2}$.

In the proof of Theorem \ref{theo:hypF} it is shown that $\Dom(K^*)\subset D_{RN}$. But one
can check that $D_{RN}=\Dom\left(Q^{1/2}\right)$, so, we have $\Dom(K^*)\subset
\Dom\left(Q^{1/2}\right)=\Dom\left(Z^{1/2}\right)$.

Finally, we get $\bigcap_{k\ge1}R_k\big(\Dom(K^{*2})\big)$ by the fact that the $R_k$s are
proportional to $a$, $a^\dagger$, $b$ or $b^\dagger$ and that $\Dom(K^{*2})\subset
\Dom(a^\dagger a)\cap \Dom(b^\dagger )$, as one can check.

For $u\in D$ we get by direct computations
\[
-2 \RE \langle u |K^* u\rangle =\sum_{k\geq 1}\norm{R_ku}^2= \norm{Z^{1/2}u}^2, \qquad
\norm{Z^{1/2}u}\leq \norm{Q^{1/2}u},
\]
\[
\norm{Q^{1/2}u}^2=w\norm{a u}^2+2 w\norm{b u}^2 + 2\left(\kappa \overline{n} + \kappa \overline{n}_p
\right) \norm{u}^2,
\]
\begin{multline*}
2\RE \langle Q^{1/2}u|Q^{1/2}K^*u\rangle + \sum_{k\geq 1}\norm{Q^{1/2}R_ku}^2 \\
{}=2\left(\kappa \overline{n} + \kappa \overline{n}_p \right) \norm{u}^2- 2w \kappa
\norm{au}^2 - 4w \kappa_p\norm{bu}^2\leq \norm{Q^{1/2}u}^2.
\end{multline*}
Then, these inequalities can be extended to the domains required in Theorem \ref{theor_QZ} and
this ends the proof.
\end{proof}

In order to describe the two photocounters and the homodyne detector we have to specialise the
observables \eqref{observables}; what we need is to take $m=3$ and \cite{BarSpringer}
\[
X(\alpha,t) = \begin{cases}
\Lambda_{\alpha \alpha} (t), &\alpha=1,\,2,
\\ \displaystyle
\int_0^t\left(\rme^{-\rmi\left(\theta_3 - \omega_C t \right)} \, \rmd A_3(s)+
\rme^{\rmi\left(\theta_3 - \omega_C t \right)}\,\rmd A^\dagger_3(s)\right), & \alpha=3.
\end{cases}
\]
This means that the quantities in Hypothesis \ref{H:theBs} are
\begin{subequations}
\begin{equation}
c(t)=0, \quad b(t)=0, \quad B^1=|z_1\rangle\langle z_1|, \quad
B^2=|z_2\rangle\langle z_2|, \quad B^3=0,
\end{equation}
\begin{equation}
h^1(t)=h^2(t)=0, \quad h^3_i(t)=\delta_{i3} \,\rme^{\rmi\left(\theta_3 - \omega_C t \right)}.
\end{equation}
\end{subequations}
This choice trivially satisfies Hypothesis \ref{H:theBs} and the expressions of the quantities
in Definition \ref{def:Phi} become
\[
\mathsf{S}\big(k(s)\big)=\1 + \sum_{j=1}^2 \left(\rme^{\rmi k_j(s)}-1\right)|z_j\rangle\langle z_j|,
\qquad \quad\mathsf{r}_{t}(k)(s)=
1_{(0,t)}(s)\, \rmi k_3(s) h^3(s)\,,
\]
\[
\big(\mathsf{S}_t(k) g\big)(s) =
1_{(0,t)}(s) \sum_{j=1}^2 \left(\rme^{\rmi k_j(s)}-1\right) g_j(s) \,z_j +g(s),
\qquad \mathsf{r}(k;s)= \rmi k_3(s) h^3(s) .
\]

Finally, in order to describe a coherent monochromatic laser pumping the $b$-mode as in the
source term in the master equation \eqref{mastereqC}, we have to take a coherent state of the
field with $f$-function given by
\begin{equation}
f_i(t)=\delta_{i4}\, \frac{\rmi \lambda \rme^{-2\rmi \omega_C t} }
{\overline{\beta_{2}}}\,1_{(0,T)}(t).
\end{equation}
We are assuming $\beta_2\neq 0$ and we understand that $T$ is a large time (needed to have an
$L^2$-function), but that $T\to +\infty$ in the reduced characteristic operator.

In conclusion the model just described is well defined, as it satisfies all the hypotheses
introduced in this paper. Moreover, one can check that the associated formal master equation
(Eq.\ \eqref{foreward} for $k=0$) reduces to Eq.\ \eqref{mastereqC}, as we wanted.


\begin{thebibliography}{99}
\bibitem{Acc78} L. Accardi, On the quantum Feynman-Kac formula, \textit{Rend. Sem. Mat.
    Fis. Milano} \textbf{XLVIII} (1978) 135--179.

\bibitem{Bar86LNP} A. Barchielli, Stochastic processes and continual measurements in
    quantum mechanics, in \textit{Stochastic
    Processes in Classical and Quantum Systems}, Lecture Notes in Physics \textbf{262},
    eds. S. Albeverio, G. Casati, D. Merlini,
    (Springer, Berlin, 1986) pp. 14--23.

\bibitem{Bar86PR} A. Barchielli, Measurement theory and stochastic differential
    equations in quantum mechanics, \textit{Phys. Rev. A} \textbf{34} (1986) 1642--1649.

\bibitem{Bar90QO} A. Barchielli, Direct and heterodyne detection and other
    applications of quantum stochastic calculus to quantum optics, \textit{Quantum Opt.} \textbf{2}
    (1990) 423--441.

\bibitem{BarSpringer} A. Barchielli, Continual Measurements in Quantum Mechanics and
    Quantum Stochastic Calculus, in \textit{Open Quantum Systems III}, Lecture Notes in Mathematics
    \textbf{1882}, eds. S. Attal, A. Joye, C.-A. Pillet, (Springer, Berlin, 2006)
    pp. 207--291.
\bibitem{BarGSpringer} A.\ Barchielli, M.\ Gregoratti, \textit{Quantum Trajectories and
    Mesurements  in Continuous Time}, Lecture Notes in Physics
    \textbf{782} (Springer, Berlin, 2009).

\bibitem{BarL85JMP} A. Barchielli, G. Lupieri, Quantum stochastic calculus, operation
    valued stochastic processes and continual measurements in quantum mechanics, \textit{J. Math.
    Phys.} \textbf{26} (1985) 2222--2230.

\bibitem{Bel88} V. P. Belavkin, Nondemolition measurements, nonlinear filtering and
    dynamic programming of quantum stochastic processes, in
    \textit{Modelling and Control of Systems}, ed. A. Blaqui\`ere,
    Lecture Notes in Control and Information
    Sciences \textbf{121} (Springer, Berlin, 1988) pp. 245--265.
\bibitem{Belavkin} V. P. Belavkin, Measurement, Filtering and Control in Quantum Open
    Dynamical Systems, \textit{Rep. Math. Phys.} \textbf{43} (1999) 405--425.

\bibitem{Bott} C. Bottero, \textit{A Qualitative Analysis of two Master Equations in Quantum
    Optics.} PhD thesis, Department of Mathematics, Politecnico di Milano, 2008.
\bibitem{Carm08} H. J. Carmichael, \textit{Statistical Methods in
    Quantum Optics 2. Non-Classical Fields} (Springer, Berlin, 2008).

\bibitem{Castro} R. Castro, \textit{Quantum Stochastic Calculus and Continual Measurements:
    The case of Unbounded Coefficients.} PhD thesis, Department of Mathematics, Univesit\`{a} di
    Roma La Sapienza, January 2007.

\bibitem{Fagnola99} F. Fagnola, Quantum Markov Semigroups and Quantum Flows,
    \textit{Proyecciones, Journal of Mathematics} \textbf{18} (1999), no.~3, pp. 1--144.

\bibitem{Fagnola02} F. Fagnola, H-P Quantum stochastic differential equations, in N. Obata,
    T. Matsui, A. Hora (eds.), \textit{Quantum probability and White Noise Analysis}, QPPQ, XVI,
    51–96, World Sci., River Edge, NJ, 2002).

\bibitem{Fagnola06} F. Fagnola, Quantum Stochastic Differential Equations and Dilation
    of Completely Positive Semigroups, in \textit{Open Quantum Systems II},
    Lecture Notes in Mathematics \textbf{1881}, eds. S. Attal, A. Joye, C.-A. Pillet, (Springer,
    Berlin, 2006) pp. 183--220.

\bibitem{FagW03} F. Fagnola, S. J. Wills, Solving quantum stochastic differential
    equations with unbounded coefficients, \textit{J. Funct. Anal.} \textbf{198} (2003) 279--310.

\bibitem{Fri85} A. Frigerio, Covariant Markov dilations of quantum dynamical
    semigroups, \textit{Pub. RIMS Kyoto Univ.} \textbf{21} (1985) 657--675.

\bibitem{GarZ00} C. W. Gardiner, P. Zoller, \textit{Quantum Noise} (Springer, Berlin, 2000).

\bibitem{GraS87} P. Grangier, R. E. Slusher, B. Yurke, A. La Porta, Squeezed-light–enhanced
    polarization interferometer, \textit{Phys. Rev. Lett.}
    \textbf{59} (1987) 2153--2156.

\bibitem{Hud-Partha} R. L. Hudson and K. R. Parthasarathy, Quantum It\^o's formula and
    stochastic evolutions, \textit{Comm. Math Phys.} \textbf{93} (1984) 301--323.

\bibitem{Partha} K. R. Parthasarathy, \textit{An Introduction to Quantum Stochastic Calculus}
    (Birkh\"auser, Basel, 1992).

\bibitem{Pazy83} A. Pazy, \textit{Semigroups of Linear Operators and Applications to Partial
    Differential Equations} (Springer, Berlin, 1983).

\bibitem{ReedS} M. Reed, B. Simon, \textit{Methods of Modern Mathematical Physics: I
    Functional Analysis} (Academic Press, 1980).

\bibitem{Sakai} S. Sakai, $C^*$\textit{-Algebras and }$W^*$\textit{-Algebras} (Springer,
    Berlin 1971).

\bibitem{ZolG97} P. Zoller, C. W. Gardiner, Quantum noise in quantum optics: the
    stochastic Schr\"odinger equation, in \textit{Fluctuations quantiques, (Les Houches 1995)},
    eds. S. Reynaud, E. Giacobino and J. Zinn-Justin, (North-Holland, Amsterdam,
    1997) pp. 79–-136.
\bibitem{WuK86} L.-A. Wu, H. J. Kimble, J. L. Hall, H. Wu, Generation of Squeezed States by
    Parametric Down Conversion, \textit{Phys. Rev. Lett. }
    \textbf{57} (1986) 2520--2523.
\bibitem{WuK87} L.-A. Wu, M. Xiao, H. J. Kimble, Squeezed states of light from an optical
    parametric oscillator, \textit{J. Opt. Soc. Am. B} \textbf{4} (1987)
    1465--1475.

\end{thebibliography}
\end{document}